\documentclass[11pt,centertags]{amsart}

\usepackage{amsmath}
\usepackage{amsthm}
\usepackage{amssymb}
\usepackage{enumitem}
\usepackage{esint}
\usepackage{graphicx}
\usepackage{hyperref}

\numberwithin{equation}{section}

\allowdisplaybreaks

\theoremstyle{definition}
	\newtheorem{definition}{Definition}[section]
	\numberwithin{definition}{section}
	\newtheorem{remark}[definition]{Remark}
	\newtheorem{example}[definition]{Example}
	\newtheorem{convention}[definition]{Convention}
\theoremstyle{plain}
	\newtheorem{lemma}[definition]{Lemma}
	\newtheorem{proposition}[definition]{Proposition}
	\newtheorem{theorem}[definition]{Theorem}

\newcommand{\Deop}{\Delta^{\mathrm{op}}}
\newcommand{\lam}{\lambda}
\newcommand{\Om}{\Omega}

\renewcommand{\th}{\theta}

\newcommand{\C}{\mathbb{C}}

\newcommand{\R}{\mathbb{R}}
\newcommand{\Pca}{\mathcal{P}}
\newcommand{\Lca}{\mathcal{L}}

\newcommand{\SL}[2]{\rm{SL}(#1,#2)}

\newcommand{\map}[3]{#1\colon#2\rightarrow#3}
\newcommand{\deq}{\mathrel{\mathop:}=}
\newcommand{\arccot}{\mathrm{arccot}}

\begin{document}

\title{Bounded cohomology via partial differential equations, I}

\author[T. Hartnick]{Tobias Hartnick}
\address{Mathematics Department, Technion, Haifa, 32000, Israel}
\email{tobias.hartnick@gmail.com}

\author[A. Ott]{Andreas Ott}
\address{Centre for Mathematical Sciences, University of Cambridge, Wilberforce Road, Cambridge CB3 0WB, United Kingdom}
\email{a.ott@dpmms.cam.ac.uk}


\begin{abstract}
	We present a new technique that employs partial differential equations in order to explicitly construct primitives in the continuous bounded cohomology of Lie groups. As an application, we prove a vanishing theorem for the continuous bounded cohomology of $\SL{2}{\R}$ in degree four, establishing a special case of a conjecture of Monod.
\end{abstract}

\maketitle

\tableofcontents

\section{Introduction}
\label{SectionIntroduction}

Ever since Gromov's seminal paper \cite{Gromov/Volume-and-bounded-cohomology}, bounded cohomology of discrete groups has proved a useful tool in geometry, topology and group theory. In recent years the scope of bounded cohomology has widened considerably. An important step was taken by Burger and Monod, who extended the theory to the category of locally compact second countable groups under the name of continuous bounded cohomology \cite{Burger/Bounded-cohomology-of-lattices-in-higher-rank-Lie-groups,Monod/Continuous-bounded-cohomology-of-locally-compact-groups}. Not only did this lead to a breakthrough in the understanding of bounded cohomology of lattices in Lie groups \cite{Burger/Bounded-cohomology-of-lattices-in-higher-rank-Lie-groups}, but also triggered a series of discoveries in rigidity theory (e.\,g.\,\cite{Burger/Continuous-bounded-cohomology-and-applications-to-rigidity-theory,Burger/Boundary-maps-in-bounded-cohomology.-Appendix-to:-Continuous-bounded-cohomology-and-applications-to-rigidity-theory-Geom.Funct.A,Burger/Bounded-cohomology-and-totally-real-subspaces-in-complex-hyperbolic-geometry,Burger/Bounded-Kahler-class-rigidity-of-actions-on-Hermitian-symmetric-spaces,Bucher/A-dual-interpretation-of-the-Gromov-Thurston-proof-of-Mostow-rigidity-and-volume-rigidity-for-representations-of-hyperbolic-latt,Monod/Cocycle-superrigidity-and-bounded-cohomology-for-negatively-curved-spaces,Monod/Orbit-equivalence-rigidity-and-bounded-cohomology,Chatterji/The-median-class-and-superrigidity-of-actions-on-CAT0-cube-complexes,Hamenstadt/Isometry-groups-of-proper-CAT0-spaces-of-rank-one,Hamenstadt/Isometry-groups-of-proper-hyperbolic-spaces}), higher Teichm\"uller theory (e.\,g.\,\cite{Burger/Higher-Teichmuller-Spaces:-from-SL2R-to-other-Lie-groups,Burger/Surface-group-representations-with-maximal-Toledo-invariant,Burger/Tight-homomorphisms-and-Hermitian-symmetric-spaces,Ben-Simon/On-weakly-maximal-representations-of-surface-groups}) and symplectic geometry (e.\,g.\,\cite{Polterovich/Growth-of-maps-distortion-in-groups-and-symplectic-geometry,Entov/Calabi-quasimorphism-and-quantum-homology}). At the same time, our understanding of the second bounded cohomology has improved. In particular, the approach originally developed for free groups \cite{Grigorchuk/Some-results-on-bounded-cohomology} and hyperbolic groups \cite{Epstein/The-second-bounded-cohomology-of-word-hyperbolic-groups} has been extended to larger classes of groups including mapping class groups \cite{Bestvina/Bounded-cohomology-of-subgroups-of-mapping-class-groups} and acylindrically hyperbolic groups \cite{Hull/Induced-quasi-cocycles-on-groups-with-hyperbolically-embedded-subgroups,Fujiwara/Bounded-cohomology-via-quasi-trees}. Moreover, there has been some progress in constructing bounded cohomology classes in higher degree \cite{Mineyev/Bounded-cohomology-characterizes-hyperbolic-groups,Hartnick/Surjectivity-of-the-comparison-map-in-bounded-cohomology-for-Hermitian-Lie-groups,Bucher/The-norm-of-the-Euler-class,Goncharov/Geometry-of-configurations-polylogarithms-and-motivic-cohomology}. 

On the other hand, our knowledge on vanishing results for bounded cohomology in higher degree is still very poor. It was already known to Johnson \cite{Johnson/Cohomology-in-Banach-algebras} that the bounded cohomology of an amenable group vanishes in all positive degrees. Here the primitive of a given cocycle is obtained by applying an invariant mean. In contrast, for non-amenable groups no general technique for constructing primitives is available so far. In particular, there is not a single non-amenable group whose bounded cohomology is known in all degrees. Actually, the situation is even worse. One may define the \emph{bounded cohomological dimension} of a group $\Gamma$ to be
\[
{\rm bcd}(\Gamma) \deq	 \sup \bigl\{ n \,\big|\, H_{b}^{n}(\Gamma; \R) \neq 0 \bigr\},
\]
where $H_{b}^{n}(\Gamma; \R)$ is the $n$-th bounded cohomology of $\Gamma$ with coefficients in the trivial module~$\R$. At present we do not even know whether there exists any group $\Gamma$ with ${\rm bcd}(\Gamma) \not \in \{ 0, \infty \}$.

The few vanishing results we have for the bounded cohomology of non-amenable groups are all based on the vanishing of all cocycles in the respective degree in some resolution. The most far-reaching results in this direction were achieved in \cite{Monod/On-the-bounded-cohomology-of-semi-simple-groups-S-arithmetic-groups-and-products} by choosing efficient resolutions. However, no such resolutions are known for dealing with the continuous bounded cohomology $H^n_{cb}(H; \R)$ of non-amenable connected Lie groups $H$. For such groups the most efficient resolution that is presently available is the boundary resolution \cite{Ivanov/Foundations-of-the-theory-of-bounded-cohomology,Burger/Bounded-cohomology-of-lattices-in-higher-rank-Lie-groups}. In this particular resolution cocycles vanish only in degree at most three; in degree greater than three there will inevitably be nonzero cocycles, and one faces the problem of finding primitives. This explains why the few existing vanishing results such as in \cite{Burger/Bounded-cohomology-of-lattices-in-higher-rank-Lie-groups,Burger/On-and-around-the-bounded-cohomology-of-SL2} do not go beyond degree three.

\smallskip

Our goal in this article is to develop a new approach to the construction of primitives in continuous bounded cohomology for real semisimple Lie groups. To demonstrate its effectiveness we settle the following special case of a conjecture due to Monod \cite[Problem A]{Monod/An-invitation-to-bounded-cohomology}.

\begin{theorem} \label{TheoremMain}
	Let $G$ be a connected real Lie group that is locally isomorphic to ${\rm SL}_2(\R)$. Then
\[
H^{4}_{cb}(G; \R) = 0.
\]
\end{theorem}

Actually, for such $G$ Monod conjectured that ${\rm bcd}(G) = 2$, which means that $H^{n}_{cb}(G; \R) = 0$ for all $n > 2$. In degree $n=3$ there are no nonzero cocycles in the boundary resolution at all \cite{Burger/On-and-around-the-bounded-cohomology-of-SL2}, but this is no longer true for $n>3$. In this sense, Theorem \ref{TheoremMain} is the prototype of a vanishing theorem that requires the construction of primitives. We believe that our method of proof generalizes to arbitrary $n>3$, and possibly to other Lie groups. This will be addressed in future work.

\smallskip

Monod's conjecture about the bounded cohomology of ${\rm SL}_2(\R)$ is a special case of a more general conjecture, which would allow to compute the continuous bounded cohomology of arbitrary connected Lie groups. In fact, since continuous bounded cohomology is invariant under division by the amenable radical \cite{Burger/Bounded-cohomology-of-lattices-in-higher-rank-Lie-groups,Monod/Continuous-bounded-cohomology-of-locally-compact-groups}, it is sufficient to compute the continuous bounded cohomology of semsimple Lie groups $H$ without compact factors and with finite center. For such groups it is conjectured \cite{Dupont/Bounds-for-characteristic-numbers-of-flat-bundles, Monod/An-invitation-to-bounded-cohomology} that the natural comparison map between the continuous bounded cohomology and the continuous cohomology is an isomorphism. This would imply that ${\rm bcd}(H)$ coincides with the dimension of the associated symmetric space, thereby providing examples of groups of arbitrary bounded cohomological dimension. Plenty is known by now about surjectivity of the comparison map \cite{Dupont/Bounds-for-characteristic-numbers-of-flat-bundles,Gromov/Volume-and-bounded-cohomology,Bucher-Karlsson/Finiteness-properties-of-characteristic-classes-of-flat-bundles,Lafont/Simplicial-volume-of-closed-locally-symmetric-spaces-of-non-compact-type,Goncharov/Geometry-of-configurations-polylogarithms-and-motivic-cohomology,Hartnick/Surjectivity-of-the-comparison-map-in-bounded-cohomology-for-Hermitian-Lie-groups}, while injectivity still remains mysterious in higher degrees. Indeed, injectivity has so far been established only in degree two for arbitrary~$H$ \cite{Burger/Bounded-cohomology-of-lattices-in-higher-rank-Lie-groups}, and for some rank one groups in degree three \cite{Burger/On-and-around-the-bounded-cohomology-of-SL2,Bloch/Higher-regulators-algebraic-K-theory-and-zeta-functions-of-elliptic-curves,Bucher/128}. Theorem \ref{TheoremMain} is the first result in degree greater than three. Incidentally, it has an application to the existence of solutions to perturbations of the Spence-Abel functional equation for Rogers' dilogarithm, along the lines suggested in \cite{Burger/On-and-around-the-bounded-cohomology-of-SL2}. This will be discussed in the forthcoming article \cite{Hartnick/Perturbations-of-the-Spence-Abel-equation-and-deformations-of-the-dilogarithm-function}.

\smallskip

For the proof of Theorem \ref{TheoremMain} we shall reformulate the problem of constructing bounded primitives in terms of a fixed point problem for the action of $G$ on a certain function space. The main idea is then to describe the fixed points as solutions of a certain system of linear first order partial differential equations. In this way, we obtain primitives by solving the corresponding Cauchy problem. We show that for carefully chosen initial conditions, particular solutions have additional discrete symmetries, which we finally use to deduce boundedness. We will give a more detailed outline of our strategy of proof in Section \ref{SubSectionStrategyOfProof} after introducing some notation.

\medskip

\textbf{Acknowledgement.} We are most indebted to B.\,Tugemann for generously conducting numerous computer experiments at an early stage of this project, and for many helpful discussions. We further thank U.\,Bader, M.\,Bj\"orklund, M.\,Bucher, M.\,Burger, A.\,Derighetti, J.\,L.\,Du-pont, H.\,Hedenmalm, A.\,Iozzi, A.\,Karlsson, H.\,Kn\"orrer, A.\,Laptev, N.\,Monod, A.\,Nevo, E.\,Sayag, I.\,Smith, J.\,Swoboda, and A.\,Wienhard for discussions and suggestions. We are very grateful to the Institut Mittag-Leffler and the organizers of the program ``Geometric and Analytic Aspects of Group Theory'', the Max Planck Institute for Mathematics and the organizers of the program ``Analysis on Lie Groups'', and \'Ecole Polytechnique F\'ed\'erale de Lausanne for their hospitality and for providing us with excellent working conditions.

T.\,H.\,received funding from the European Research Council under the European Union's Seventh Framework Programme (FP7/2007-2013), ERC Grant agreement no.\,306706.

A.\,O.\,wishes to thank the Institut des Hautes \'Etudes Scientifiques, the Isaac Newton Institute for Mathematical Sciences, the Max Planck Institute for Mathematics, the Centre de Recerca Matem\`atica Barcelona, and the Hausdorff Research Institute for Mathematics for their hospitality and excellent working conditions. He was supported by a grant from the Klaus Tschira Foundation, by grant EP/F005431/1 from the Engineering and Physical Sciences Research Council, and partially supported by grant ERC-2007-StG-205349 from the European Research Council.

\section{Primitives in continuous bounded cohomology}
\label{SectionPrimitivesInContinuousBoundedCohomology}

\subsection{The boundary model of continuous bounded cohomology}
\label{SubSubSectionTheBoundaryModelOfContinuousBoundedCohomology}

We collect some basic facts about the continuous bounded cohomology of Lie groups. By \cite[Cor.\,7.5.10]{Monod/Continuous-bounded-cohomology-of-locally-compact-groups} continuous bounded cohomology is invariant under local isomorphisms, whence it suffices to prove Theorem~\ref{TheoremMain} for the group $G \deq {\rm PU}(1,1)$. Recall that $G$ acts by fractional linear transformations on the closed unit disc $\overline{\mathbb D} \subset \C$. We define $K \deq {\rm Stab}_{G}(0)$ and $P \deq {\rm Stab}_{G}(1)$. Then $K \subset G$ is a maximal compact subgroup and $P \subset G$ is a parabolic.

The boundary model of the continuous bounded cohomology of $G$ relies on the action of $G$ on the boundary $S^1 = \partial\,\mathbb D$ of the unit disc by fractional linear transformations, which we shall refer to as the \emph{boundary action}. Note that this action is strictly $3$-transitive, since it is conjugate via the Cayley transform to the ${\rm PSL}_{2}(\R)$-action on the real projective line.

Throughout we denote by $\mu_K$ the unique $K$-invariant probability measure on $S^1$. We shall write $\mathcal{M}((S^{1})^{n})$ for the space of  $\mu_K$-measurable real-valued functions on $(S^{1})^{n}$ and $\mathcal{L}^{\infty}((S^{1})^{n})$ for the subspace of bounded functions. The quotients of these spaces obtained by identifying~$\mu_K$-almost everywhere coinciding functions will be denoted by $M((S^{1})^{n})$ and $L^{\infty}((S^{1})^{n})$, respectively. The diagonal $G$-action on $(S^1)^n$ induces actions on all of these spaces. For any subgroup $H \subset G$ we denote by $\mathcal M((S^{1})^{n})^{H}$, $\mathcal L^\infty((S^{1})^{n})^{H}$, $M((S^{1})^{n})^{H}$ and $L^\infty((S^1)^{n})^{H}$ the corresponding subspaces of $H$-invariants. Observe that the homogeneous differential 
\[
\map{d^{n}}{M((S^{1})^{n})}{M((S^{1})^{n+1})}, \quad df(z_{0}, \ldots, z_{n}) = \sum_{j=0}^{n} (-1)^{j} \, f(z_{0}, \ldots, \widehat{z_{j}}, \ldots, z_{n})
\]
maps $L^{\infty}((S^{1})^{n})^{H}$ to $L^\infty((S^{1})^{n+1})^{H}$. Tailoring the general boundary model due to Burger and Monod \cite{Burger/Bounded-cohomology-of-lattices-in-higher-rank-Lie-groups,Monod/Continuous-bounded-cohomology-of-locally-compact-groups} to our needs, the boundary model of the continuous bounded cohomology of $G$ takes on the following form.

\begin{proposition}[Boundary model, {\cite[Thm.\,7.5.3]{Monod/Continuous-bounded-cohomology-of-locally-compact-groups}}]
	The continuous bounded cohomology of~$G$ is given by the cohomology of the complex $(L^{\infty}((S^{1})^{\bullet+1})^{G}, d)$:
\[
H^{n}_{cb}(G; \R) \cong H^{n}\bigl( L^{\infty}( (S^{1})^{\bullet+1} )^{G}, d \bigr)
\]
for all $n \geq 0$.
\end{proposition}

\subsection{Primitives}
\label{SubSubSectionPrimitives}

For any cocycle $c \in L^{\infty}( (S^{1})^5 )^{G}$, i.e.\,$dc=0$, and any closed subgroup $H \subset G$, we denote by 
\[
\Pca(c)^{H} \deq \left\{ p \in M\bigl( (S^{1})^{4} \bigr)^{H} \,\Big|\, dp = c \right\}
\]
the space of $H$-invariant primitives of $c$ and by
\[
\Pca^{\infty}(c)^{H} \deq \left\{ p \in L^{\infty}\bigl( (S^{1})^{4} \bigr)^{H} \,\big|\, dp = c \right\}
\]
the subspace of bounded $H$-invariant primitives. Proving Theorem \ref{TheoremMain} then amounts to showing that $\Pca^{\infty}(c)^{G} \neq \emptyset$ for any such cocycle $c$. We will achieve this by explicitly constructing bounded $G$-invariant primitives. For this purpose, we first define an operator
\[
\map{I}{\Lca^{\infty}\bigl( (S^{1})^{n+1} \bigr)}{\Lca^{\infty}\bigl( (S^{1})^{n} \bigr)}
\]
by
\begin{equation} \label{EquationOperatorI}
	I(c)(z_{1},\dots,z_{n}) \deq \int_{S^{1}} c(z,z_{1}\dots,z_{n}) \, d\mu_{K}(z).
\end{equation}
It induces an operator $\map{I}{L^{\infty}((S^{1})^{n+1})}{L^{\infty}((S^{1})^{n})}$, which by abuse of notation we denote by the same symbol. By integrating the cocycle equation $dc=0$, we see that $d(I(c)) = c$ for every cocycle $c$.

Let us now fix a cocycle $c \in L^{\infty}((S^{1})^{5})^{G}$. By $K$-invariance of the measure $\mu_{K}$ we see from formula \eqref{EquationOperatorI} that the function $I(c)$ is $K$-invariant. It will, however, in general not be~$G$-invariant. In order to obtain~$G$-invariant primitives we amend the operator $I$ in the following way.

We denote by $(S^{1})^{(n)} \subset (S^{1})^{n}$ the subset of $n$-tuples of pairwise distinct points in $S^{1}$. Note in particular that the $G$-action on $(S^{1})^{(3)}$ is free and has two open orbits given by positively and negatively oriented triples. We write $C^{\infty}((S^{1})^{(3)})^{K}$ for the space of $K$-invariant real-valued smooth functions on~$(S^{1})^{(3)}$ and consider it as a subspace of $M((S^{1})^{3})$. We then define an operator
\[
\map{P_{c}}{C^{\infty}\bigl( (S^{1})^{(3)} \bigr)^{K}}{M\bigl( (S^{1})^{4} \bigr)}, \quad f \mapsto I(c) + df.
\]
A key observation is that all $G$-invariant bounded primitives of $c$ necessarily lie in the image of the operator $P_{c}$. This will allow us to express primitives in terms of smooth (rather than measurable) solutions to differential equations.

\pagebreak

\begin{proposition} \label{PropositionParametrization}
	The image of the operator $P_{c}$ satisfies
\[
\Pca^{\infty}(c)^{G} \subset P_{c}\bigl( C^{\infty}( (S^{1})^{(3)})^{K} \bigr) \subset \Pca(c)^{K}.
\]
\end{proposition}

We will give a proof using the existence of invariant representatives for bounded $G$-invariant function classes. Here by an invariant representative of a function class $q \in L^\infty((S^{1})^{n})^{G}$ we mean a preimage $\tilde{q} \in \mathcal{L}^{\infty}((S^{1})^{n})^{G}$. We have the following existence result from \cite[Thm.\,A]{Monod/Equivariant-measurable-lifting}.

\begin{lemma}[Invariant representatives] \label{LemmaInvariantRepresentatives}
	The map $\mathcal{L}^{\infty}((S^{1})^{n})^{G} \to L^{\infty}((S^{1})^{n})^{G}$ is surjective.
\end{lemma}

We now use this lemma to prove the proposition.

\begin{proof}[Proof of Proposition \ref{PropositionParametrization}]
	We have already seen above that $I(c) \in \Pca^{\infty}(c)^{K}$. We conclude that $P_{c}(f) \in \Pca(c)^{K}$ for all $f \in C^{\infty}( (S^{1})^{(3)})^{K}$. It also follows that if $p \in \mathcal P^\infty(c)^G$ is any bounded primitive then $d (p - I(c)) = 0$. Hence we must have $p = I(c) + df$ for some \emph{measurable} function~$f$ given by $f = I(p - I(c))$. The nontrivial part of the proof is to show that this function $f$ can be chosen to be smooth. In fact, by the lemma we may take invariant representatives $\tilde{p}$ and $\tilde{c}$ and set
\[
f \deq I(\tilde{p} - I(\tilde{c})) \in \mathcal{M}((S^{1})^{3})^{K}.
\]
We claim that this function is smooth on $(S^1)^{(3)}$. Indeed, for every $g \in G$ we have 
\begin{multline*}
	f(g.z_{1}, g.z_{2}, g.z_{3}) = \int_{S^{1}} \Big( \tilde{p}(z, g.z_{1}, g.z_{2}, g.z_{3}) - I(\tilde{c})(z, g.z_{1}, g.z_{2}, g.z_{3}) \Big) d\mu_{K}(z) \\ = \int_{S^{1}} \int_{S^{1}} \frac{d(g.\mu_{K})}{d\mu_{K}}(z) \left( \tilde{p}(z, z_{1}, z_{2}, z_{3}) - \frac{d(g.\mu_{K})}{d\mu_{K}}(w) \, \tilde{c}(z, w, z_{1}, z_{2}, z_{3}) \right) d\mu_{K}(w) \, d\mu_{K}(z).
\end{multline*}
Smoothness of the Radon-Nikodym derivative \cite[Prop.\,8.43]{Knapp/Lie-groups-beyond-an-introduction} hence implies that the function on $G$ given by $g \mapsto f(g.z_{1}, g.z_{2}, g.z_{3})$ is smooth. Since $G$-orbits are open in $(S^1)^{(3)}$ we infer that the function $f$ is smooth.
\end{proof}

\subsection{Strategy of proof}
\label{SubSectionStrategyOfProof}

We briefly outline the strategy for the proof of Theorem \ref{TheoremMain}. We shall proceed in three steps.
\begin{enumerate}[itemsep=1ex, leftmargin=1cm, itemindent=0cm, labelwidth=0cm, labelsep=0.2cm]
	\item In Section \ref{SectionPartialDifferentialEquations}, we show that for any function $f \in C^{\infty}( (S^{1})^{(3)})^{K}$ the primitive $P_{c}(f)$ is $G$-invariant if and only if $f$ satisfies a certain system of linear first order partial differential equations.
	\item In Section \ref{SectionConstructionOfPrimitives}, we explicitly construct solutions $f$ of this system of differential equations, showing that $\Pca(c)^{G} \neq \emptyset$.
	\item In Section \ref{SectionBoundednessOfPrimitives}, we prove that there exist particular solutions $f$ with certain additional discrete symmetries. For such functions $f$ we then show boundedness of $P_{c}(f)$, establishing that $\Pca^{\infty}(c)^{G} \neq \emptyset$.
\end{enumerate}

\section{Partial differential equations}
\label{SectionPartialDifferentialEquations}

\subsection{The boundary action in local coordinates}
\label{SubSectionTheBoundaryActionInLocalCoordinates}

Recall that elements of ${\rm SU}(1,1)$ are matrices of the form 
\[
g_{a,b} \deq \left( \begin{matrix} a & b \\ \bar b &\bar a \end{matrix}
\right)
\]
for numbers $a,b \in \C$ satisfying $|a|^2-|b|^2 = 1$. We denote by $[g_{a,b}]$ the corresponding elements in $G = {\rm PU}(1,1)$. If we define
\[
k_{\xi} \deq \bigl[ g_{e^{i\xi/2},\,0} \bigr], \quad a_{s} \deq \left[ g_{\cosh(-s/2),\,\sinh(-s/2)} \right], \quad n_{t} \deq \left[ g_{1+\frac{i}{2} t,\,- \frac{i}{2} t} \right],
\]
then the maps $\xi \mapsto k_{\xi}$, $s \mapsto a_{s}$ and $t \mapsto n_{t}$ are one-parameter groups $\R \to G$ whose images we denote by $K$, $A$ and $N$, respectively. Note that this is compatible with our definition of the subgroup $K$ in Section \ref{SubSubSectionTheBoundaryModelOfContinuousBoundedCohomology}, and that $A$ and $N$ are the Levi factor and unipotent radical of the parabolic $P$. In particular, $\mathrm{Fix}(A)=\{\pm1\}$ and $\mathrm{Fix}(N)=\{1\}$. Moreover, we obtain an Iwasawa decomposition $G = KAN$, and every elliptic (hyperbolic, parabolic) element in $G$ is conjugate to an element in $K$ ($A$, $N$).

The coordinate on the boundary $S^{1} \subset \C$ will be denoted by a complex number $z$ of modulus~$1$. In addition, it will often be convenient to work with the angular coordinate $\th \in [0, 2\pi)$ defined by $z = e^{i \th}$. Correspondingly, on $(S^1)^{n}$ we will use the two sets of coordinates $(z_{0},\ldots,z_{n-1})$ and $(\th_{0},\ldots,\th_{n-1})$. Note that, in angular coordinates on $S^{1}$, the measure~$\mu_{K}$ is given by
\[
\int_{S^{1}}f(z) \, d\mu_{K}(z) = \fint f(\th) \, d\th \deq \frac{1}{2\pi} \int_{0}^{2\pi} f(\th) \, d\th.
\]

\begin{convention} \label{ConventionAngularCoordinates}
	Throughout, all operations on angular coordinates will implicitly be understood modulo~$2\pi$. For example, $\th_{2}-\th_{1}$ denotes the unique point in the interval $[0,2\pi)$ that is congruent to $\th_{2}-\th_{1}$ modulo~$2\pi$.
\end{convention}

The $G$-action on $S^{1}$ induces a $G$-action on the interval $[0, 2\pi)$ by the relation $g.e^{i\th} = e^{i \, g.\theta}$. Note that in particular $k_{\xi}.\th = \th + \xi$. The next lemma provides formulas for the infinitesimal action of the one-parameter subgroups $\{a_{s}\}$ and $\{n_{t}\}$ in angular coordinates.
\begin{lemma} \label{LemmaInfinitesimalANAction}
	The infinitesimal action of $\{a_{s}\}$ and $\{n_{t}\}$ in angular coordinates is given by
\[
\frac{d}{ds}(a_{s}.\eta) = \sin(a_{s}.\eta), \quad \frac{d}{dt}(n_{t}.\eta) = 1-\cos(n_{t}.\eta)
\]
for $\eta \in [0,2\pi)$.
\end{lemma}

\begin{proof}
	To prove the first formula, we compute
\[
\left.\frac{d}{ds}\right|_{s=0}(a_{s}.\phi) = \frac{1}{i \, e^{i\phi}} \left.\frac{d}{ds}\right|_{s=0}  e^{i \, a_{s}.\phi} =  \frac{1}{i \, e^{i\phi}} \left.\frac{d}{ds}\left(\frac{\cosh(-s/2) \, e^{i\phi} + \sinh(-s/2)}{\sinh(-s/2) \, e^{i\phi} + \cosh(-s/2)}\right) \right|_{s=0} = \sin(\phi).
\]
Since $\{a_{s}\}$ is a one-parameter group we further infer that
\[
\frac{d}{ds}(a_{s}.\eta) = \left.\frac{d}{d\sigma} \right|_{\sigma=0}(a_{s+\sigma}.\eta) = \left.\frac{d}{d\sigma} \right|_{\sigma=0}(a_{\sigma}(a_{s}.\eta)) = \sin(a_{s}.\eta).
\]
Likewise, for the second formula we compute
\[
\left.\frac{d}{dt} \right|_{t=0}(n_{t}.\phi) = \frac{1}{i \, e^{i\phi}} \left.\frac{d}{dt} \right|_{t=0} e^{i \, n_{t}.\phi} = \frac{1}{i \, e^{i\phi}} \left.\frac{d}{dt}\left(\frac{(1 + \frac{i}{2} t) \, e^{i\phi} - \frac{i}{2} t}{\frac{i}{2} t \, e^{i\phi} + 1 - \frac{i}{2} t} \right) \right|_{t=0} = 1-\cos(\phi)
\]
and conclude as above.
\end{proof}

We have chosen the somewhat non-standard parametrization of the groups $A$ and $N$ above in such a way that the infinitesimal action of $\{a_{s}\}$ and $\{n_{t}\}$ is given by the simple formulas of the lemma.

\subsection{Fundamental vector fields}
\label{SubSectionFundamentalVectorFields}

We denote by $L_{K}^{(n)}$, $L_{A}^{(n)}$ and $L_{N}^{(n)}$ the differential operators that appear as fundamental vector fields for the infinitesimal action of the one-parameter groups $\{k_{\xi}\}$, $\{a_{s}\}$ and $\{n_{t}\}$ on $(S^{1})^{(n)}$. By Lemma \ref{LemmaInfinitesimalANAction}, they are given in angular coordinates by
\[
\begin{split}
	L_{K}^{(n)} &= \sum_{j=0}^{n-1} \frac{\partial}{\partial \th_{j}} \\
	L_A^{(n)} &=  \sum_{j=0}^{n-1} \sin(\th_j) \, \frac{\partial}{\partial \th_{j}} \\
	L_N^{(n)} &= \sum_{j=0}^{n-1} (1-\cos(\th_{j})) \, \frac{\partial}{\partial \th_{j}}.
\end{split}
\]
The next lemma is crucial for applications of the operators $L_{K}^{(n)}$, $L_{A}^{(n)}$ and $L_{N}^{(n)}$ in cohomology.

\begin{lemma} \label{LemmaOperatorsCommuteWithDifferentials}
	Let $L^{(n)}$ denote one of the operators $L_{K}^{(n)}$, $L_{A}^{(n)}$ and $L_{N}^{(n)}$. Then $L^{(n)}$ commutes with the homogeneous differential in the sense that
\[
d^{n} \circ L^{(n)}  = L^{(n+1)} \circ d^{n}
\]
for every $n>0$.
\end{lemma}

\begin{proof}
	Let $\lam \in C^\infty([0, 2\pi))$ and consider the differential operators
\[
L_{\lam}^{(n)} \deq \sum_{j=0}^{n-1} \lam(\th_{j}) \, \frac{\partial}{\partial \th_{j}}
\]
for every $n>0$. For any smooth $(n-1)$-cochain $q$ we compute
\[
\begin{split}
	\left( L_{\lam}^{(n+1)} \bigl( d^{n} q \bigr) \right) (\th_{0}, \dots, \th_{n}) &=\sum_{j=0}^{n} \lam(\th_{j}) \, \frac{\partial}{\partial \th_{j}} \left( \sum_{\ell=0}^{n} (-1)^\ell \, q\bigl( \th_{0}, \dots, \widehat{\th_{\ell}}, \dots, \th_{n} \bigr) \right) \\
	&= \sum_{\ell=0}^{n} (-1)^\ell \, \sum_{j\neq \ell} \lam(\th_{j}) \, \frac{\partial q}{\partial \th_{j}} \bigl( \th_{0}, \dots, \widehat{\th_{\ell}}, \dots, \th_{n} \bigr) \\
	&= \sum_{\ell=0}^{n} (-1)^\ell \, \bigl( L_{\lam}^{(n)}q \bigr) \bigl( \th_{0}, \dots, \widehat{\th_{\ell}}, \dots, \th_{n} \bigr) \\
	&= \left( d^{n} \bigl( L_{\lam}^{(n)}q \bigr) \right) \! (\th_{0}, \dots, \th_{n}). \qedhere
\end{split}
\]
\end{proof}

\subsection{Infinitesimal invariance of primitives}

We are now in a position to characterize $G$-invariance of primitives $P_{c}(f)$ in terms of differential equations for the function $f$. First, let us define functions $c^{\sharp}$ and $c^{\flat}$ by
\begin{equation} \label{EquationDefinitionOfcsharp}
	c^{\sharp}(\th_{0}, \th_{1}, \th_{2}) \deq \fint\fint \cos(\varphi) \, c(\eta, \varphi, \theta_0, \theta_1, \theta_2) \, d\eta \, d\varphi
\end{equation}
and
\begin{equation} \label{EquationDefinitionOfcflat}
	c^{\flat}(\th_{0}, \th_{1}, \th_{2}) \deq \fint\fint \sin(\varphi) \, c(\eta, \varphi, \th_{0}, \th_{1}, \th_{2}) \, d\eta \, d\varphi.
\end{equation}
An argument as in the proof of Proposition \ref{PropositionParametrization} using smoothness of the Radon-Nikodym derivative shows that we may choose the functions $c^{\sharp}$ and $c^{\flat}$ to be smooth on $(S^{1})^{(3)}$. The next proposition achieves the first step in the agenda outlined in Section \ref{SubSectionStrategyOfProof}.

\begin{proposition} \label{PropositionPDEForf}
	Let $f \in C^{\infty}((S^{1})^{(3)})^{K}$. Then the primitive $P_{c}(f) \in \Pca(c)^{K}$ is $G$-invariant if and only if there exist functions  $v^{\sharp}, v^{\flat} \in C^{\infty}((S^{1})^{(2)})$ such that the triple $(f,v^{\sharp},v^{\flat})$ satisfies the system of partial differential equations
\begin{equation} \label{EquationSystemForf}
\left\{
\begin{aligned}
	L_{A}^{(3)}f &= c^{\sharp} + dv^{\sharp} \\
	L_{N}^{(3)}f &= c^{\flat} + dv^{\flat}.
\end{aligned}
\right.
\end{equation}
\end{proposition}

Note that all functions appearing in \eqref{EquationSystemForf} are smooth, so all derivatives can be understood classically. The proof of Proposition \ref{PropositionPDEForf} relies on the following lemma.

\begin{lemma} \label{LemmaOperatorsActOnPrimitive}
	The function $I(c)$ is smooth along $G$-orbits and satisfies
\[
\begin{split}
	L_{A}^{(4)} I(c) &= - dc^{\sharp} \\
	L_{N}^{(4)} I(c) &= - dc^{\flat},
\end{split} 
\]
where $c^\sharp$ and $c^\flat$ are as in \eqref{EquationDefinitionOfcsharp} and \eqref{EquationDefinitionOfcflat}.
\end{lemma}

\begin{proof}
	An argument as in the proof of Proposition \ref{PropositionParametrization} shows that the function $I(c)$ can be chosen to be smooth along $G$-orbits. Fix an invariant representative $\tilde{c}$ of $c$ by Lemma \ref{LemmaInvariantRepresentatives}. Using $A$-invariance of $\tilde{c}$ and Lemma \ref{LemmaInfinitesimalANAction} we compute
\[
\begin{split}
	L_{A}^{(4)}(I(c))(\theta_0, \dots, \theta_{3}) &= \left. \frac{d}{ds} \right|_{s=0} \fint \tilde{c}(\varphi, a_{s}.\th_{0}, \dots, a_{s}.\th_{3} ) \, d\varphi \\
	&=	\left. \frac{d}{ds} \right|_{s=0} \fint \frac{d(a_{s}.\varphi)}{d\varphi} \, \tilde{c}(\varphi, \th_{0}, \dots, \th_{3}) \, d\varphi \\
	&=	\fint \frac{d}{d\varphi} \left.\frac{d(a_{s}.\varphi)}{ds} \right|_{s=0} \tilde{c}(\varphi, \th_{0}, \dots, \th_{3}) \, d\varphi \\
	&= \left. \fint \frac{d}{d\varphi} \sin(a_{s}.\varphi) \right|_{s=0} \tilde{c}(\varphi, \th_{0}, \dots, \th_{3}) \, d\varphi \\
	&= \fint \cos(\varphi) \, c(\varphi, \th_{0}, \dots, \th_{3}) \, d\varphi.
\end{split}
\]
On the other hand, the cocycle identity
\[
\begin{split}
	0 &= dc(\eta, \varphi, \th_{0}, \dots, \th_{3}) \\
	&= c(\varphi,\th_{0}, \dots, \th_{3}) - c(\eta,\th_{0}, \dots, \th_{3}) + \sum_{j=0}^{3} (-1)^j \, c(\eta, \varphi, \th_{0}, \dots, \widehat{\th_{j}}, \dots, \th_{3})
\end{split}
\]
integrates against $\cos(\varphi)$ to

\pagebreak

\[
\begin{split}
	0 &= \fint \fint \cos(\varphi) \, c(\varphi,\th_{0}, \dots, \th_{3}) \, d\eta \, d\varphi - 0 + \sum_{j=0}^{3} (-1)^j \, c^{\sharp}(\th_{0}, \dots, \widehat{\th_{j}}, \dots, \th_{3}) \\
	&= \fint \cos(\varphi) \, c(\varphi,\th_{0}, \dots, \th_{3}) \, d\varphi + d{c}^{\sharp}(\th_{0}, \dots, \th_{3}).
\end{split}
\]
This establishes the first identity. Likewise, to prove the second identity we compute
\[
\begin{split}
	L_{N}^{(4)}(I(c))(\th_{0}, \dots, \th_{3}) &= \left. \frac{d}{dt} \right|_{t=0} \fint \tilde{c}(\varphi, n_{t}.\th_{0}, \dots, n_{t}.\th_{3}) \, d\varphi \\
	&= \fint \frac{d}{d\varphi} \left. \frac{d(n_{t}.\varphi)}{dt} \right|_{t=0} \tilde{c}(\varphi, \th_{0}, \dots, \th_{3}) \, d\varphi \\
	&= \left. \fint \frac{d}{d\varphi} \bigl( 1 - \cos(n_{t}.\varphi) \bigr) \right|_{t=0} \tilde{c}(\varphi, \th_{0}, \dots, \th_{3}) \, d\varphi \\
	&= \fint \sin(\varphi) \, \tilde{c}(\varphi, \th_{0}, \dots, \th_{3}) \, d\varphi,
\end{split}
\]
and then integrate the above cocycle identity against $\sin(\varphi)$.
\end{proof}

\begin{proof}[Proof of Proposition \ref{PropositionPDEForf}]
	First of all, we observe that for $f \in C^{\infty}((S^{1})^{(3)})^{K}$ we may choose the function
\[
P_{c}(f) = I(c) + df
\]
to be smooth along $G$-orbits. This follows by an argument as in the proof of Proposition \ref{PropositionParametrization}. Hence the primitive $P_{c}(f)$ is $G$-invariant if and only if it is infinitesimally $G$-invariant. Since the subgroups $K$, $A$ and $N$ generate the group $G$, this in turn is equivalent to $P_{c}(f)$ satisfying the system of partial differential equations
\[
\left\{
\begin{aligned}
	L_{K}^{(4)} P_{c}(f) &= 0 \\
	L_{A}^{(4)} P_{c}(f) &= 0 \\
	L_{N}^{(4)} P_{c}(f) &= 0.
\end{aligned}
\right.
\]
The first equation is automatically satisfied since $P_{c}(f)$ is $K$-invariant by Proposition \ref{PropositionParametrization}. Writing out the definition of $P_{c}(f)$ and applying Lemmas \ref{LemmaOperatorsCommuteWithDifferentials} and \ref{LemmaOperatorsActOnPrimitive}, the remaining two equations are seen to be equivalent to
\[
\left\{
\begin{aligned}
	d\bigl( L_{A}^{(3)}f \bigr) &= dc^{\sharp} \\
	d\bigl( L_{N}^{(3)}f \bigr) &= dc^{\flat}.
\end{aligned}
\right.
\]
Applying the operator $I$ as in the proof of Proposition \ref{PropositionParametrization}, we conclude that $f\in C^\infty((S^{1})^{(3)})^{K}$ solves this system if and only if there exist \emph{measurable} functions $v^{\sharp}$ and $v^{\flat}$ such that the triple $(f, v^{\sharp}, v^{\flat})$ satisfies system \eqref{EquationSystemForf}. However, it is not difficult to see that $v^{\sharp}$ and $v^{\flat}$ can be chosen to be smooth. In fact, by \eqref{EquationSystemForf} the functions $dv^{\sharp}$ and $dv^{\flat}$ are differences of smooth functions and hence smooth. Now replace $v^{\sharp}$ and $v^{\flat}$ by $I(dv^{\sharp})$ and $I(dv^{\flat})$.
\end{proof}

\subsection{The Frobenius integrability condition}

We now turn to the problem of finding solutions of system \eqref{EquationSystemForf}. As we shall see in Proposition \ref{PropositionIntegrabilityCondition} below, it follows from the classical Frobenius theorem that this system admits a smooth solution $(f, v^{\sharp}, v^{\flat})$ if and only if the functions $v^{\sharp}$ and $v^{\flat}$ satisfy a certain integrability condition. In order to state the result we need to introduce the function
\begin{equation} \label{EquationDefinitionOfccheck}
	\check{c}(\phi_{1},\phi_{2}) \deq \fint \fint \fint \sin(\eta-\varphi) \, c(\eta, \varphi, \psi, \phi_{1}, \phi_{2}) \, d\eta \, d\varphi \, d\psi.
\end{equation}

\begin{lemma} \label{LemmaKInvarianceOfcchek}
	The function $\check{c}$ is smooth on $(S^{1})^{(2)}$ and $K$-invariant.
\end{lemma}

\begin{proof}
	An argument as in the proof of Proposition \ref{PropositionParametrization} shows that $\check{c} \in C^{\infty}((S^{1})^{(2)})$. By $K$-invariance of $c$ and of the measure, for every $\xi \in [0,2\pi]$ we have
\[
\begin{split}
	\check{c}(k_{\xi}.\phi_{1}, k_{\xi}.\phi_{2}) &= \fint \fint \fint \sin(\eta-\varphi) \, c(\eta, \varphi, \psi, \phi_{1} + \xi, \phi_{2} + \xi) \, d\eta \, d\varphi \, d\psi \\
	&= \fint \fint \fint \sin(\eta-\varphi) \, c(\eta-\xi, \varphi-\xi, \psi-\xi, \phi_{1}, \phi_{2}) \, d\eta \, d\varphi \, d\psi \\
	&= \fint \fint \fint \sin \bigl( (\eta+\xi)-(\varphi+\xi) \bigr) \, c(\eta, \varphi, \psi, \phi_{1}, \phi_{2} ) \, d\eta \, d\varphi \, d\psi = \check{c}(\phi_{1}, \phi_{2}). \qedhere
\end{split}
\]
\end{proof}

\begin{proposition}[Integrability condition] \label{PropositionIntegrabilityCondition}
	The system \eqref{EquationSystemForf} admits a solution $(f,v^{\sharp},v^{\flat})$ if and only if the pair $(v^{\sharp},v^{\flat})$ satisfies the system of partial differential equations
\begin{equation} \label{EquationFrobeniusSystem}
	\left\{
	\begin{aligned}
		d\bigl( L_{K}^{(2)} v^{\sharp} + v^{\flat} \bigr) &= 0 \\
		d\bigl( L_{K}^{(2)} v^{\flat} - v^{\sharp} \bigr) &= 0 \\
		d\bigl( L_{K}^{(2)} v^{\sharp} - L_{N}^{(2)} v^{\sharp} + L_{A}^{(2)} v^{\flat} - \check{c} \bigr) &= 0. \\
	\end{aligned}
	\right.
\end{equation}
\end{proposition}

The proof of the proposition relies on the Frobenius theorem and will be deferred to Appendix~\ref{SectionTheFrobeniusIntegrabilityCondition}. In the sequel, we shall refer to system \eqref{EquationFrobeniusSystem} as the \emph{Frobenius system}. It will be enough for our purposes to find a function $f$ satisfying system \eqref{EquationSystemForf} for some pair $(v^{\sharp},v^{\flat})$. We will hence not attempt to find all solutions of the system \eqref{EquationFrobeniusSystem}. Rather, we will construct a single special solution $(v^{\sharp},v^{\flat})$.

\begin{proposition} \label{PropositionPDEForr}
	Let $r \in C^{\infty}((0, 2\pi),\C)$ be a smooth complex-valued solution of the ordinary differential equation
\begin{equation} \label{EquationODEForr}
	(1 - e^{-i\phi}) \cdot \frac{d r}{d\phi} = i \, r(\phi) - \check{c}(0, \phi).
\end{equation}
By Convention \ref{ConventionAngularCoordinates}, we may define a function $v \in C^{\infty}((S^{1})^{(2)},\C)$ by
\begin{equation} \label{EquationDefinitionOfv}
	v(\th_{1}, \th_{2}) \deq e^{i\th_{1}} \, r(\th_{2} - \th_{1}).
\end{equation}
Then the pair $(v^{\sharp}, v^{\flat}) \deq ({\rm{Re}}(v), {\rm{Im}}(v))$ is a solution of the Frobenius system \eqref{EquationFrobeniusSystem}.
\end{proposition}

\begin{proof}
	We remind the reader that throughout the proof we adhere to Convention \ref{ConventionAngularCoordinates}. First observe that if $v$ is a solution of the system
\begin{equation} \label{EquationComplexFrobeniusSystem}
	\left\{
	\begin{aligned}
		L^{(2)}_{K} v &= i \, v \\
		\displaystyle
		e^{-i \th_{1}} \frac{\partial v}{\partial \th_{1}} + e^{-i \th_{2}} \frac{\partial v}{\partial \th_{2}} &= \check{c},
	\end{aligned}
	\right.
\end{equation}
then $(v^{\sharp}, v^{\flat}) \deq ({\rm Re}(v), {\rm Im}(v))$ is a solution of \eqref{EquationFrobeniusSystem}. Indeed, taking real and imaginary parts of the first equation in \eqref{EquationComplexFrobeniusSystem} we obtain
\[
L_{K}^{(2)} v^{\sharp} =  - v^{\flat}, \quad L_{K}^{(2)}v^{\flat} = v^{\sharp},
\]
while taking the real part of the second equation yields
\[
L_{K}^{(2)} v^{\sharp} - L_{N}^{(2)} v^{\sharp} + L_{A}^{(2)} v^{\flat} = \check{c}.
\]
Next consider the transformation
\[
u(\th_{1}, \th_{2}) \deq e^{-i\th_{1}} \, v(\th_{1}, \th_{2}).
\]
Then the first equation in \eqref{EquationComplexFrobeniusSystem} is equivalent to
\begin{equation} \label{EquationFirstForu}
	L^{(2)}_{K} u = 0,
\end{equation}
and the second equation is equivalent to
\begin{eqnarray} \label{EquationSecondForu}
\begin{aligned}
	\check{c}(\th_{1}, \th_{2}) &= e^{-i\theta_1} \, \frac{\partial}{\partial \th_{1}} \left(e^{i\th_{1}} \, u(\th_{1}, \th_{2}) \right) + e^{-i \th_{2}} \, \frac{\partial}{\partial \th_{2}} \left(e^{i\th_{1}} \, u(\th_{1}, \th_{2}) \right) \\
	&= i \, u(\th_{1}, \th_{2}) + \frac{\partial u}{\partial \th_{1}} + e^{i (\th_{1}-\th_{2})} \, \frac{\partial u}{\partial \th_{2}} \\
	&= i \, u(0, \th_{2}-\th_{1}) + \bigl( 1 - e^{i(\th_{1}-\th_{2})} \bigr) \, \frac{\partial u}{\partial \th_{1}}.
\end{aligned}
\end{eqnarray}
Here we used that $\partial_{\th_{1}}u = - \partial_{\th_{2}}u$ by \eqref{EquationFirstForu}. Let now $r$ be a solution of equation \eqref{EquationODEForr} and set $v(\th_{1}, \th_{2}) \deq e^{i\th_{1}} \, r(\th_{2} - \th_{1})$. Then $u(\th_{1}, \th_{2}) = r(\th_{2}-\th_{1})$ obviously satisfies \eqref{EquationFirstForu}. By $K$-invariance of $\check{c}$ from Lemma \ref{LemmaKInvarianceOfcchek} we have $\check{c}(\th_{1}, \th_{2}) = \check{c}(0, \th_{2}-\th_{1})$. It follows that $u$ solves equation \eqref{EquationSecondForu}.
\end{proof}

\section{Construction of primitives}
\label{SectionConstructionOfPrimitives}

\subsection{Solving the Frobenius system}

The first step in the construction of the primitive $P_{c}(f)$ is to solve the differential equation \eqref{EquationODEForr} for the function $r$. As we have seen in Propositions~\ref{PropositionPDEForr} and \ref{PropositionIntegrabilityCondition}, the function $r$ then gives rise to a special solution $(v^{\sharp},v^{\flat})$ of the Frobenius system \eqref{EquationFrobeniusSystem}, which in turn determines the inhomogeneities in system \eqref{EquationSystemForf} in such a way that this system admits a solution $f$.

The complex ordinary differential equation \eqref{EquationODEForr} can be solved by applying the method of variation of constants. Its general solution $r \in C^{\infty}((0, 2\pi),\C)$ is given by
\[
r(\phi) = \bigl( 1-e^{i\phi} \bigr) \cdot \left( C_{0} - \frac{1}{2} \int_{\pi}^{\phi} \frac{\check{c}(0,\zeta)}{1-\cos(\zeta)} \, d\zeta \right),
\]
where $C_{0}$ is an arbitrary complex constant. Note that different choices of $C_{0}$ lead to cohomologous cochains $v$. We may therefore assume $C_{0} = 0$, obtaining
\begin{equation} \label{EquationDefinitionOfr}
	r(\phi) = - \frac{1}{2} \bigl( 1-e^{i\phi} \bigr) \cdot \int_{\pi}^{\phi} \frac{\check{c}(0,\zeta)}{1-\cos(\zeta)} \, d\zeta.
\end{equation}
We shall henceforth be working with the function $r$ defined by this formula. A crucial observation is the following lemma.

\pagebreak

\begin{lemma}[Boundedness of the inhomogeneity] \label{LemmaBoundednessOfrAndw}
	The function $r$ is bounded. In particular, the inhomogeneities in system \eqref{EquationSystemForf} are bounded.
\end{lemma}

\begin{proof}
	Observe that $\| \check{c} \|_{\infty} \le \| c \|_{\infty}$ and $\sqrt{1-\cos(\phi)} \cdot |\cot(\phi/2)| = \sqrt{2} \cdot |\cos(\phi/2)|\leq \sqrt{2}$. Hence for all $\phi \in (0,2\pi)$ we have an estimate
\[
\begin{split}
	|r(\phi)| &\leq \frac{1}{2} \cdot \| \check{c} \|_{\infty} \cdot  |1-e^{i\phi}| \cdot \left|\int_{\pi}^\phi \frac{1}{1-\cos(\zeta)} \, d\zeta\right| \\
	&= \frac{\sqrt{2}}{2} \cdot \| \check{c} \|_{\infty} \cdot \sqrt{1-\cos(\phi)} \cdot |\cot(\phi/2)| \\
	&\le \| c \|_{\infty}. \qedhere
\end{split}
\]
\end{proof}

\subsection{Cauchy initial value problem}
\label{SubSectionCauchyInitialValueProblem}

Recall that solutions to a first order linear partial differential equation may be constructed explicitly by integration along its characteristic curves, with initial values prescribed on some non-characteristic hypersurface \cite[Ch.\,3]{Caratheodory/Calculus-of-variations-and-partial-differential-equations-of-the-first-order.-Part-I:-Partial-differential-equations-of-the-firs}. We shall now apply this principle in order to explicitly construct solutions of the system \eqref{EquationSystemForf}.

Let the function $r$ be given by formula \eqref{EquationDefinitionOfr}. By Proposition \ref{PropositionPDEForr} and Lemma \ref{LemmaBoundednessOfrAndw} this determines a bounded solution $(v^{\sharp}, v^{\flat}) \deq ({\rm{Re}}(v), {\rm{Im}}(v))$ of the Frobenius system \eqref{EquationFrobeniusSystem} by
\begin{equation} \label{EquationDefinitionOfv2}
	v(\th_{1}, \th_{2}) \deq e^{i\th_{1}} \, r(\th_{2} - \th_{1}).
\end{equation}
We shall henceforth keep the function $v$ defined that way. By Proposition \ref{PropositionIntegrabilityCondition} the system \eqref{EquationSystemForf} then admits a solution $f \in C^{\infty}((S^{1})^{(3)})$ satisfying the system of equations
\begin{subequations}
\begin{align}
	L_{K}^{(3)}f &= 0 \label{EquationCauchyK} \\
	L_{A}^{(3)}f &= c^\sharp + dv^\sharp \label{EquationCauchyA} \\
	L_{N}^{(3)}f &= c^\flat + dv^\flat. \label{EquationCauchyN}
\end{align}
\end{subequations}
We know from Section \ref{SubSectionFundamentalVectorFields} that the characteristic curves for these equations are precisely the orbits for the actions of the one-parameter groups $K=\{k_{\xi}\}$, $A=\{a_{s}\}$ and $N=\{n_{t}\}$ on $(S^{1})^{(3)}$, respectively. In order to construct the function $f$ we may therefore proceed as follows.
\begin{enumerate}[itemsep=1ex, leftmargin=1cm, itemindent=0cm, labelwidth=0cm, labelsep=0.2cm]
	\item We use equation \eqref{EquationCauchyK} in order to construct $f$ with initial values $f_{0} \deq f|_{H_{0}}$ prescribed on the hypersurface
\[
H_{0} \deq \bigl\{ (z_{0},z_{1},z_{2}) \in (S^{1})^{(3)} \,\big|\, z_{0} \deq 1 \bigr\}.
\]
Of course, equation \eqref{EquationCauchyK} just says that $f$ is constant along the $K$-orbits in $(S^{1})^{(3)}$. The hypersurface $H_{0}$ is non-characteristic for the equation \eqref{EquationCauchyK} since it intersects transversally with the $K$-orbits in $(S^{1})^{(3)}$. Note that in order for $f$ to be compatible with the remaining equations \eqref{EquationCauchyA} and \eqref{EquationCauchyN}, the hypersurface $H_{0}$ has to be invariant under the actions of~$A$ and $N$. This, however, is indeed the case since the point $1$ remains fixed under these two actions.	\item We use equation \eqref{EquationCauchyN} in order to construct $f_{0}$ with initial values $f_{1} \deq f_{0}|_{H_{1}}$ prescribed on the union of $A$-orbits
\[
H_{1} \deq \bigl\{ a_{s}.(1,i,-i) \,\big|\, -\infty < s < \infty \bigr\} \cup \bigl\{ a_{s}.(1,-i,i) \,\big|\, -\infty < s < \infty \bigr\} \subset H_{0}.
\]
Note that $H_{1}$ is non-characteristic for the equation \eqref{EquationCauchyN} since it intersects transversally with the $N$-orbits in $H_{0}$. Moreover, in order for $f_{0}$ to be compatible with the remaining equation \eqref{EquationCauchyA}, the curve $H_{1}$ has to be invariant under the action of $A$, which is obviously the case.
	\item We use equation \eqref{EquationCauchyA} in order to construct $f_{1}$ with initial values $f_{2} \deq f_{1}|_{H_{2}}$ prescribed on the set of \emph{base points}
\[
H_{2} \deq \bigl\{ \bigl( 1,e^{2\pi i/3},e^{4\pi i/3} \bigr),\bigl( 1,e^{4\pi i/3},e^{2\pi i/3} \bigr) \bigr\} \subset H_{1}.
\]
Note that equation \eqref{EquationCauchyA}, when restricted to the curve $H_{1}$, becomes an ordinary differential equation which can be solved directly.
\end{enumerate}

In particular, we see that solutions of system \eqref{EquationSystemForf} are uniquely determined by the initial values of $f_{2}$ on the set of base points $H_{2}$ and hence form a $2$-parameter family. We will work out the details of (1) in Section \ref{SubSectionReductionOfVariables}, while the details of (2) and (3) will be worked out in Section~\ref{SubSectionMethodOfCharacteristics}.

\subsection{Reduction of variables}
\label{SubSectionReductionOfVariables}

In angular coordinates, the hypersurface $H_{0}$ introduced in the previous subsection is given by
\[
H_{0} = \bigl\{ (\th_{0},\th_{1},\th_{2}) \in [0,2\pi)^{3} \,\big|\, \th_{0}=0, \th_{1}\neq\th_{2}, \th_{1}\neq0\neq\th_{2} \bigr\}.
\]
The canonical projection $(\th_{0},\th_{1},\th_{2}) \mapsto (\th_{1},\th_{2})$ identifies $H_{0}$ with the domain $\Om \deq (0, 2\pi)^2 \setminus \Delta$, where $\Delta \subset (0, 2\pi)^2$ denotes the diagonal in the open square. The coordinates in $\Om$ will be denoted by $(\phi_{1},\phi_{2})$. Moreover, we write $\Om_{\pm} \deq \{ (\phi_{1},\phi_{2}) \,|\, \phi_{1} \lessgtr \phi_{2} \}$ for the open subsets corresponding to the two $G$-orbits in~$(S^{1})^{(3)}$ consisting of positively and negatively oriented triples. The restriction $f_{0} \deq f|_{H_{0}}$ is then given by $f_{0}(\phi_{1},\phi_{2}) = f(0,\phi_{1},\phi_{2})$. Note that by $K$-invariance the function $f$ is recovered from~$f_{0}$ by
\begin{equation} \label{EquationCorrespondencefAndf0}
	f(\th_{0}, \th_{1}, \th_{2}) = f_{0}(\th_{1}-\th_{0}, \th_{2}-\th_{0}).
\end{equation}
Since the hypersurface $H_{0}$ is invariant under the actions of $A$ and $N$, the system \eqref{EquationSystemForf} restricts to the system
\begin{equation} \label{EquationSystemForf0}
	\left\{
	\begin{aligned}
		L_{A}^{(2)} f_{0} &= c^{\sharp}_{0} + (d v^\sharp)_{0} \\
		L_{N}^{(2)} f_{0} &= c^{\flat}_{0} + (d v^\flat)_{0}.
	\end{aligned}
	\right.
\end{equation}
Here we denote by $c^{\sharp}_{0}$, $c^{\flat}_{0}$, $(d v^\sharp)_{0}$ and $(d v^\flat)_{0}$ the respective restrictions of the functions $c^{\sharp}$, $c^{\flat}$, $d v^\sharp$ and $d v^\flat$, i.\,e.\,,
\begin{equation} \label{EquationDefinitionOfc_0}
	c^{\sharp}_{0}(\phi_{1},\phi_{2}) \deq c^{\sharp}(0,\phi_{1},\phi_{2}), \quad c^{\flat}_{0}(\phi_{1},\phi_{2}) \deq c^{\flat}(0,\phi_{1},\phi_{2})
\end{equation}
and
\begin{equation} \label{EquationDefinitionOfdv_0}
	(d v^\sharp)_{0}(\phi_{1},\phi_{2}) \deq d v^\sharp(0,\phi_{1},\phi_{2}), \quad (d v^\flat)_{0}(\phi_{1},\phi_{2}) \deq d v^\flat(0,\phi_{1},\phi_{2}).
\end{equation}
Note that these functions are smooth on $\Om$.

\subsection{Method of characteristics}
\label{SubSectionMethodOfCharacteristics}

We now construct the function $f_{0}$ by integrating equations \eqref{EquationCauchyA} and \eqref{EquationCauchyN} along their characteristic curves. Recall that the characteristics for these equations are precisely the orbits for the actions of the one-parameter groups $A=\{a_{s}\}$ and $N=\{n_{t}\}$ on the domain $\Om$ (see Figure \ref{FigureOrbits}).
\begin{figure}
	\includegraphics{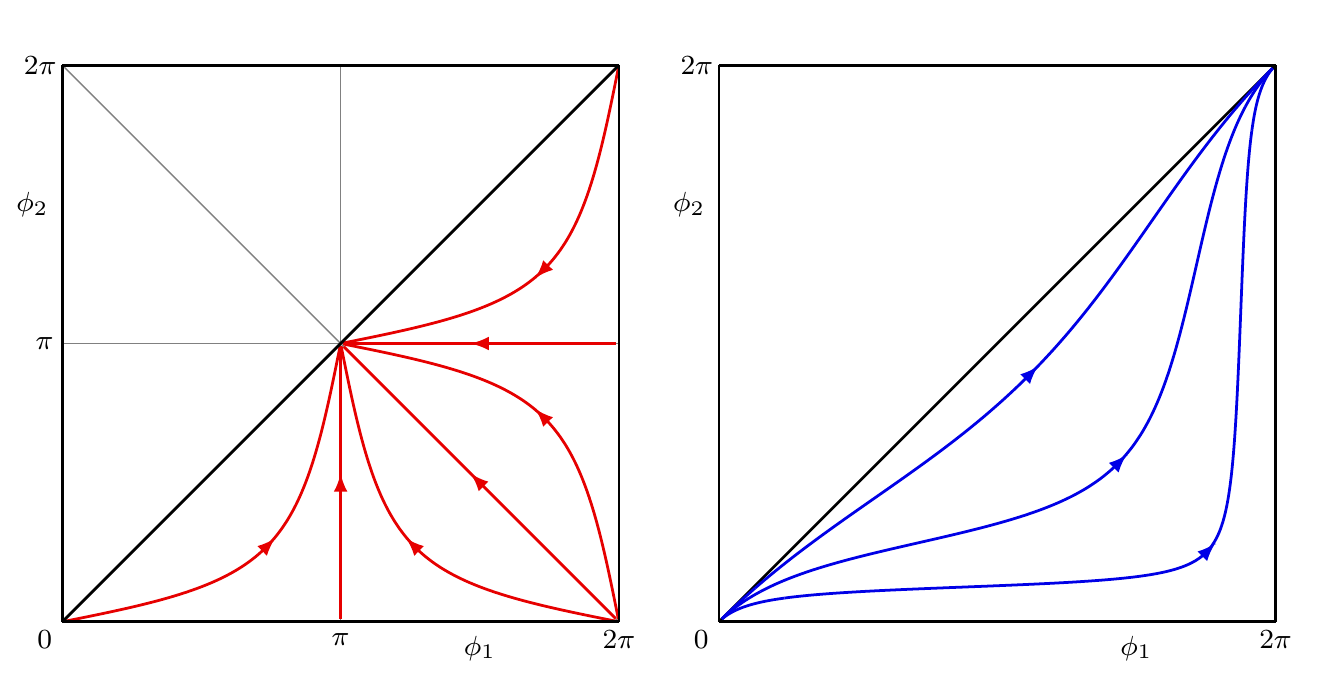}
	\centering
	\caption{$A$-orbits (left) and $N$-orbits (right) in the domain $\Om_{-}$}
	\label{FigureOrbits}
\end{figure}
It will be convenient to abbreviate the inhomogeneities appearing on the right-hand side of system \eqref{EquationSystemForf0} by
\begin{equation} \label{EquationInhomogeneitiesF}
	F^{\sharp}_{c} \deq c^{\sharp}_{0} + (dv^{\sharp})_{0} \quad \text{and} \quad F^{\flat}_{c} \deq c^{\flat}_0 + (dv^{\flat})_{0}.
\end{equation}
If we prescribe the value of $f_{0}$ on a single point in each of the orbits $\Om_{+}$ and $\Om_{-}$, the function~$f_{0}$ will be uniquley determined by the relations
\[
f_{0}(a_{S}.\phi_{1}, a_{S}.\phi_{2}) - f_{0}(\phi_{1}, \phi_{2}) = \int_{0}^{S} F^{\sharp}_{c}(a_{s}.\phi_{1},a_{s}.\phi_{2}) \, ds
\]
and
\[
f_{0}(n_{T}.\phi_{1},n_{T}.\phi_{2}) - f_{0}(\phi_{1},\phi_{2}) = \int_{0}^{T} F^{\flat}_{c}(n_{t}.\phi_{1},n_{t}.\phi_{2}) \, dt.
\]
More precisely, let us denote by
\[
\Deop \deq \bigl\{ (\phi, 2\pi-\phi) \,\big|\, \phi \in (0, 2\pi) \setminus \{ \pi \} \bigr\} \subset \Om
\]
the antidiagonal in $\Om$, which corresponds to the hypersurface $H_{1}$ introduced in Section \ref{SubSectionCauchyInitialValueProblem}. Note that it has two connected components. In order to compute the function $f_{0}$ we first introduce new coordinates on $\Om$ that are adapted to the $N$-orbits. For every point $(\phi_{1},\phi_{2}) \in \Om$ we define $\Phi(\phi_{1},\phi_{2}) \in (0,\pi) \cup (\pi,2\pi)$ in such a way that $( \Phi(\phi_{1}, \phi_{2}), 2\pi-\Phi(\phi_{1},\phi_{2}) )$ is the point of intersection of the antidiagonal with the unique $N$-orbit passing through the point $(\phi_{1},\phi_{2})$. We then define $T(\phi_{1}, \phi_{2}) \in (-\infty,\infty)$ by the relation
\[
(\phi_{1},\phi_{2}) = n_{T(\phi_{1},\phi_{2})}.\bigl( \Phi(\phi_{1},\phi_{2}), 2\pi-\Phi(\phi_{1},\phi_{2}) \bigr).
\]
For later reference we note that
\begin{equation} \label{EquationDefinitionOfT}
	T(\phi_{1},\phi_{2}) = -\frac{1}{2} \left( \cot\left(\frac{\phi_{1}}{2}\right) + \cot\left(\frac{\phi_{2}}{2}\right) \right).
\end{equation}
Here we used that the action of the one-parameter subgroup $\{ n_{t} \}$ is given by the formula $n_{t}.\phi = 2 \, \cot(-t + \arccot(\phi/2))$. Integrating the second equation in \eqref{EquationSystemForf0} along the $N$-orbits in~$\Om$, with initial values $f_{1}$ prescribed on the antidiagonal $\Deop$, we then obtain

\pagebreak

\begin{multline} \label{IntegralFormulaForf0}
	f_{0}(\phi_{1},\phi_{2}) = f_{1}\bigl( (\Phi(\phi_{1},\phi_{2}),2\pi-\Phi(\phi_{1},\phi_{2}) \bigr) \\ + \int_{0}^{T(\phi_{1}, \phi_{2})} F^{\flat}_{c} \bigl( n_{t}.\Phi(\phi_{1}, \phi_{2}), n_{t}.(2\pi-\Phi(\phi_{1}, \phi_{2})) \bigr) \, dt
\end{multline}
for every $(\phi_{1},\phi_{2}) \in \Om$. It remains to compute the function $f_{1}$ along the antidiagonal. Let
\[
\omega_{+}\deq(2\pi/3,4\pi/3) \quad \text{and} \quad \omega_{-}\deq(4\pi/3,2\pi/3)
\]
be the points in $\Om$ corresponding to the base points in $H_{2}$ introduced in Section~\ref{SubSectionCauchyInitialValueProblem}. Note that $\omega_{+}$ and $\omega_{-}$ coincide with the barycenters of the triangles enclosing the domains $\Om_{+}$ and $\Om_{-}$ (see Figure \ref{FigurePath}). Define a new coordinate~$S(\phi) \in (-\infty,\infty)$ on each component of the antidiagonal~$\Deop$ by the relation
\[
(\phi,2\pi-\phi) = a_{S(\phi)}.\omega_{\pm},
\]
depending on whether the point $(\phi,2\pi-\phi)$ lies in $\Om_{+}$ or $\Om_{-}$.
\begin{figure}
	\includegraphics{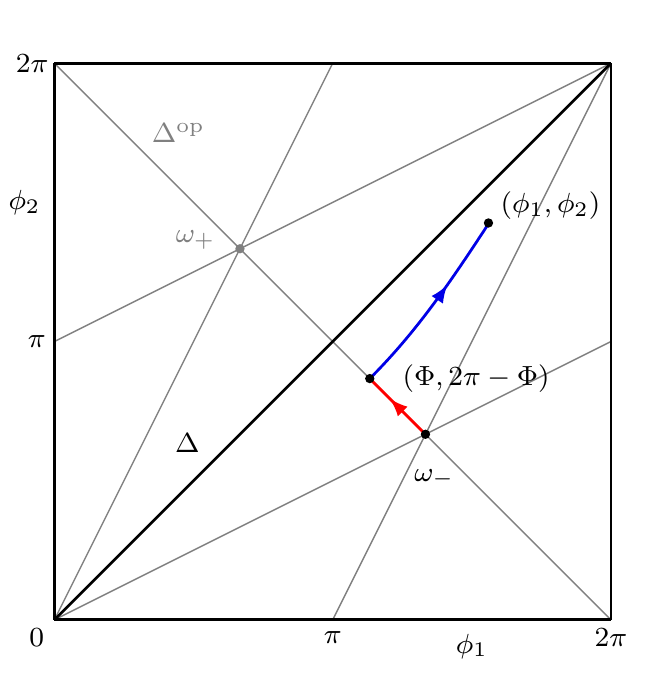}
	\centering
	\caption{A path traveling along $A$- and $N$-orbits from the basepoint $\omega_{-}$ to some point $(\phi_{1},\phi_{2})$ in $\Om_{-}$}
	\label{FigurePath}
\end{figure}
Integrating the first equation in \eqref{EquationSystemForf0} along the $A$-orbits in $\Deop$, with initial values $f_{2}$ prescribed on the base points $\{\omega_{+},\omega_{-}\}$, we get
\begin{equation} \label{IntegralFormulaForf1}
	f_{1}( \phi,2\pi-\phi ) = f_{2}(\omega_{\pm}) + \int_{0}^{S(\phi)} F^{\sharp}_{c} (a_{s}.\omega_{\pm}) \, ds
\end{equation}
for every $\phi \in (0,\pi) \cup (\pi,2\pi)$.

\subsection{Explicit primitives}

Combining the results from the previous subsections, we are now in a position to give the following explicit characterization of primitives.

\pagebreak

\begin{proposition}[Explicit primitives] \label{PropositionExplicitPrimitives}
	Let $v^{\sharp}$ and $v^{\flat}$ be the real and imaginary parts of the function $v$ defined by formula \eqref{EquationDefinitionOfv2}, where $r$ is as in \eqref{EquationDefinitionOfr}. Then the following hold.
\begin{enumerate}[itemsep=1ex, leftmargin=1cm, itemindent=0cm, labelwidth=0cm, labelsep=0.2cm, label=(\roman{enumi})]
	\item Let $f \in C^{\infty}((S^{1})^{(3)})^{K}$. The primitive $P_{c}(f) \in \Pca(c)^{K}$ is $G$-invariant if and only if the function $f$ solves system \eqref{EquationSystemForf}.
	\item There is a one-to-one correspondence between solutions $f \in C^{\infty}((S^{1})^{(3)})^{K}$ of system \eqref{EquationSystemForf} and solutions $f_{0} \in C^{\infty}(\Om)$ of system \eqref{EquationSystemForf0} via the relation
\[
f(\th_{0}, \th_{1}, \th_{2}) = f_{0}(\th_{1}-\th_{0}, \th_{2}-\th_{0}).
\]
	\item Every pair $(f_{0}(\omega_{+}),f_{0}(\omega_{-})) \in \R^{2}$ of initial values uniquely determines a smooth solution $f_{0} \in C^{\infty}(\Om)$ of system \eqref{EquationSystemForf0} by the formula
\begin{multline} \label{EquationFormulaForf0}
	\hspace{5mm} f_{0}(\phi_{1},\phi_{2}) = f_{0}(\omega_{\pm}) + \int_{0}^{S(\Phi(\phi_{1}, \phi_{2}))} F^{\sharp}_{c} (a_{s}.\omega_{\pm}) \, ds \\ + \int_{0}^{T(\phi_{1}, \phi_{2})} F^{\flat}_{c} \bigl( n_{t}.\Phi(\phi_{1}, \phi_{2}), n_{t}.(2\pi-\Phi(\phi_{1}, \phi_{2})) \bigr) \, dt,
\end{multline}
where the functions $F^{\sharp}_{c}$ and $F^{\flat}_{c}$ are as in \eqref{EquationInhomogeneitiesF}. Conversely, any smooth solution of system \eqref{EquationSystemForf0} arises in this way.
\end{enumerate}
\end{proposition}

\begin{proof}
	Let $f \in C^{\infty}((S^{1})^{(3)})^{K}$. Assertion (i) holds by Proposition \ref{PropositionPDEForf}, while (ii) was proved in Section \ref{SubSectionReductionOfVariables}. Finally, our considerations in Section \ref{SubSectionMethodOfCharacteristics} show that the function $f_{0}$ satisfies \eqref{EquationSystemForf0} if and only if it is given by formulas \eqref{IntegralFormulaForf0} and \eqref{IntegralFormulaForf1}, in terms of integration along the unique path in the domain $\Om$ starting at~$\omega_{\pm}$ and traveling to $(\phi_{1},\phi_{2})$ along $A$- and $N$-orbits via the point $(\Phi(\phi_{1},\phi_{2}),2\pi-\Phi(\phi_{1},\phi_{2})$ on the antidiagonal (see Figure \ref{FigurePath}), with initial values prescribed at~$\omega_{\pm}$. This proves (iii).
\end{proof}

The proposition achieves the second step in the agenda outlined in Section \ref{SubSectionStrategyOfProof}. In particular, it shows that solutions $f_{0}$ of system \eqref{EquationSystemForf0} form a $2$-parameter family.

\section{Boundedness of primitives}
\label{SectionBoundednessOfPrimitives}

\subsection{Symmetries}

In the last section we saw that $G$-invariant primitives of the cocycle~$c$ come in $2$-parameter families. We will prove in Section \ref{SubSectionBoundedness} below that any such primitive is bounded if it obeys certain additional discrete symmetries. These symmetries are intimately related to symmetries of the cocycle~$c$ itself. We remind the reader of the following basic fact \cite[Scholium 7.4.6]{Monod/Continuous-bounded-cohomology-of-locally-compact-groups}.

\begin{lemma} \label{LemmaAlternatingCocycle}
	Every cocycle is cohomologous to an alternating cocycle.
\end{lemma}

Hence for the proof of Theorem \ref{TheoremMain} we may without loss of generality assume that the cocycle~$c$ is alternating. The next lemma establishes a basic symmetry property for such cocycles.

\begin{lemma} \label{LemmaSymmetryForCocycle}
	Assume that the cocycle $c$ is alternating. Then
\[
c(z_{0},z_{1},z_{2},z_{3},z_{4}) = c\bigl( z_{0}^{-1},z_{1}^{-1},z_{2}^{-1},z_{3}^{-1},z_{4}^{-1} \bigr).
\]
\end{lemma}

\begin{proof}
	Let us denote by
\[
\map{C}{\R}{S^{1}}, \quad x \mapsto \frac{x-i}{x+i}
\]
the Cayley transform. Recall that the cross-ratio
\[
(w_{0},w_{1};w_{2},w_{3}) = \frac{(w_{0}-w_{2})(w_{1}-w_{3})}{(w_{1}-w_{2})(w_{0}-w_{3})}
\]
is invariant under the action of $G$ on $S^{1}$ and is normalized in the sense that
\[
w = C\bigl( (1,-1;-i,w) \bigr).
\]
Since $c$ is alternating, we may without loss of generality assume that the triple $(z_{0}, z_{1}, z_{2})$ is positively oriented. Hence it follows by $3$-transitivity of the $G$-action on $S^{1}$ that for every point $(z_{0}, z_{1}, z_{2}, z) \in (S^{1})^{(4)}$ there exists a unique $g \in G$ such that
\[
g.(z_{0}, z_{1}, z_{2}, z) = \bigl( 1,-1,-i,C(z_{0},z_{1};z_{2},z) \bigr).
\]
Observe that the cross-ratio has the obvious symmetry
\[
(w_{1},w_{2};w_{3},w_{4}) = \bigl( w_{1}^{-1},w_{2}^{-1};w_{3}^{-1},w_{4}^{-1} \bigr).
\]
Let now $\tilde{c}$ be an invariant representative of the function $c$ as in Lemma \ref{LemmaInvariantRepresentatives}. Then the previous two identities yield
\[
\begin{split}
	\tilde{c}(z_{0},z_{1},z_{2},z_{3},z_{4}) &=  \tilde{c} \bigl( 1,-1,-i,C((z_{0},z_{1};z_{2},z_{3})),C((z_{0},z_{1};z_{2},z_{4})) \bigr) \\
	&=  \tilde{c} \bigl( 1,-1,-i,C\bigl(\bigl(z_{0}^{-1},z_{1}^{-1};z_{2}^{-1},z_{3}^{-1}\bigr)\bigr),C\bigl(\bigl(z_{0}^{-1},z_{1}^{-1};z_{2}^{-1},z_{4}^{-1}\bigr)\bigr) \bigr) \\
	&= \tilde{c} \bigl( z_{0}^{-1}, z_{1}^{-1}, z_{2}^{-1}, z_{3}^{-1},z_{4}^{-1} \bigr)
\end{split}
\]
by $G$-invariance of $\tilde{c}$.
\end{proof}
 
\begin{remark}
	More generally, a similar symmetry as in Lemma \ref{LemmaSymmetryForCocycle} holds for all bounded $G$-invariant measurable functions in $L^{\infty}((S^{1})^{n+1})^{G}$, for every $n\ge3$.
\end{remark}

With the above two lemmas at hand, we may now establish symmetries for the inhomogeneities appearing in the systems \eqref{EquationSystemForf} and \eqref{EquationSystemForf0}.

\begin{proposition}[Symmetries] \label{PropositionSymmetries}
	Assume that the cocycle $c$ is alternating. Then the inhomogeneities on the right-hand sides of systems \eqref{EquationSystemForf} and \eqref{EquationSystemForf0} have the following properties.
\begin{enumerate}[itemsep=1ex, leftmargin=1.5cm, itemindent=0cm, labelwidth=0cm, labelsep=0.2cm, label=(\roman{enumi})]
	\item The functions $c^{\sharp} + dv^{\sharp}$ and $c^{\flat} + dv^{\flat}$ are alternating.
	\item The function $F^{\sharp}_{c} = c^{\sharp}_{0} + (dv^{\sharp})_{0}$ is antisymmetric about the antidiagonal $\Deop$ in $\Om$. In particular, it vanishes along the antidiagonal.
	\item The function $F^{\flat}_{c} = c^{\flat}_{0} + (dv^{\flat})_{0}$ is symmetric about the antidiagonal $\Deop$ in $\Om$.
\end{enumerate}
\end{proposition}

\begin{proof}
	We begin with the following observation. The function $\check{c}$ defined in \eqref{EquationDefinitionOfccheck} is alternating since $c$ is assumed to be alternating. By Lemma \ref{LemmaKInvarianceOfcchek}, the function $\check{c}$ is $K$-invariant. Hence
\[
\check c(0,\zeta) = \check c(-\zeta,0) = -\check c(0,-\zeta).
\]
By Convention \ref{ConventionAngularCoordinates}, substituting $\zeta$ by $2\pi-\zeta$ we infer from this that
\begin{equation} \label{EquationSymmetryForIntegral}
	\int_{\pi}^{\phi} \frac{\check{c}(0,\zeta)}{1-\cos(\zeta)} \, d\zeta = - \int_{\pi}^{2\pi-\phi} \frac{\check{c}(0,-\zeta)}{1-\cos(-\zeta)} \, d\zeta = \int_{\pi}^{-\phi} \frac{\check{c}(0,\zeta)}{1-\cos(\zeta)} \, d\zeta
\end{equation}
for every $\phi \in (0,2\pi)$. Recall moreover from Section \ref{SubSectionCauchyInitialValueProblem} that $(v^{\sharp}, v^{\flat}) \deq ({\rm{Re}}(v), {\rm{Im}}(v))$, where
\[
v(\th_{1}, \th_{2}) = e^{i\th_{1}} \, r(\th_{2} - \th_{1})
\]
and $r$ is as in \eqref{EquationDefinitionOfr}. Let us prove (i). Since $c$ is alternating, it is immediate from \eqref{EquationDefinitionOfcsharp} and \eqref{EquationDefinitionOfcflat} that $c^{\sharp}$ and $c^{\flat}$ are alternating. By \eqref{EquationDefinitionOfr} and \eqref{EquationSymmetryForIntegral} we have
\begin{eqnarray} \label{EquationAlternatingv}
\begin{aligned}
	v(\th_{1},\th_{2}) &= e^{i\th_{1}} \, r(\th_{2}-\th_{1}) \\
	&= - \frac{1}{2} \, \bigl(e^{i\th_{1}}-e^{i\th_{2}} \bigr) \cdot \int_{\pi}^{\th_{2}-\th_{1}} \frac{\check{c}(0,\zeta)}{1-\cos(\zeta)} \, d\zeta \\
	&= \frac{1}{2} \, \bigl( e^{i\th_{2}}-e^{i\th_{1}} \bigr) \cdot \int_{\pi}^{\th_{1}-\th_{2}} \frac{\check{c}(0,\zeta)}{1-\cos(\zeta)} \, d\zeta = -v(\th_{2}, \th_{1}).
\end{aligned}
\end{eqnarray}
It follows that $dv$, and hence $dv^{\sharp}$ and $dv^{\flat}$ are alternating. This proves (i). For the proof of (ii) and (iii) we have to show that
\begin{equation} \label{EquationSymmetriesForFProof}
	F^{\sharp}_{c}(\phi_{1}, \phi_{2}) = -F^{\sharp}_{c}(-\phi_{2}, -\phi_{1}), \quad F^{\flat}_{c}(\phi_{1}, \phi_{2}) = F^{\flat}_{c}(-\phi_{2}, -\phi_{1}).
\end{equation}
To this end, we first note that by \eqref{EquationDefinitionOfc_0} and Lemma \ref{LemmaSymmetryForCocycle} we have
\begin{eqnarray} \label{EquationSymmetriesForcsharp0}
\begin{aligned}
	c^{\sharp}_{0}(-\phi_{2}, -\phi_{1}) &= \fint \fint \cos(\varphi) \, c(\eta, \varphi, 0, -\phi_{2}, -\phi_{1}) \, d\eta \, d\varphi \\
	&= \fint \fint \cos(\varphi) \, c(-\eta, -\varphi, 0, \phi_{2}, \phi_{1}) \, d\eta \, d\varphi \\
	&= \fint \fint \cos(-\varphi) \, c(\eta, \varphi, 0, \phi_{2}, \phi_{1}) \, d\eta \,d\varphi \\
	&= -\fint \fint \cos(\varphi) \, c(\eta, \varphi, 0, \phi_{1}, \phi_{2}) \, d\eta \, d\varphi = -c^{\sharp}_{0}(\phi_{1}, \phi_{2}),
\end{aligned}
\end{eqnarray}
and similarly
\begin{eqnarray} \label{EquationSymmetriesForcflat0}
\begin{aligned}
	c^{\flat}_{0}(-\phi_{2}, -\phi_{1}) &= \fint \fint \sin(\varphi) \, c(\eta, \varphi, 0, -\phi_{2}, -\phi_{1}) \, d\eta \, d\varphi \\
	&= \fint \fint \sin(-\varphi) \, c(\eta, \varphi, 0, \phi_{2}, \phi_{1}) \, d\eta \,d\varphi \\
	&= \fint \fint \sin(\varphi) \, c(\eta, \varphi, 0, \phi_{1}, \phi_{2}) \, d\eta \, d\varphi = c^{\flat}_{0}(\phi_{1}, \phi_{2}).
\end{aligned}
\end{eqnarray}
Applying \eqref{EquationSymmetryForIntegral} as in the proof of (i) above, we obtain
\[
\begin{split}
	v(-\th_{1},-\th_{2}) &= - \frac{1}{2} \, \bigl(e^{-i\th_{1}} - e^{-i\th_{2}} \bigr) \cdot \int_{\pi}^{-\th_{2}+\th_{1}} \frac{\check{c}(0,\zeta)}{1-\cos(\zeta)} \, d\zeta \\
	&= - \frac{1}{2} \, \bigl(e^{-i\th_{1}} - e^{-i\th_{2}} \bigr) \cdot \int_{\pi}^{\th_{2}-\th_{1}} \frac{\check{c}(0,\zeta)}{1-\cos(\zeta)} \, d\zeta = \overline{v(\th_{1},\th_{2})}.
\end{split}
\]
Combining this with \eqref{EquationAlternatingv} we arrive at
\[
v(\th_{1},\th_{2}) = - \overline{v(-\th_{2},-\th_{1})}.
\]
Consider the function $(dv)_{0}(\phi_{1},\phi_{2}) \deq dv(0,\phi_{1},\phi_{2})$. The previous two identities imply that
\[
\begin{split}
	(dv)_{0}(\phi_{1}, \phi_{2}) &= v(\phi_{1}, \phi_{2}) - v(0,\phi_{2}) + v(0,\phi_{1}) \\
	&= - \overline{\bigl( v(-\phi_{2}, -\phi_{1}) - v(0,-\phi_{1}) + v(0,-\phi_{2}) \bigr)} = - \overline{(dv)_{0}(-\phi_{2}, -\phi_{1})}.
\end{split}
\]
Recall from \eqref{EquationDefinitionOfdv_0} that $(dv^{\sharp})_{0}$ and $(dv^{\flat})_{0}$ are the real and imaginary parts of $(dv)_{0}$. Hence we conclude that
\begin{equation} \label{EquationSymmetriesFordv0}
	(dv^{\sharp})_{0}(\phi_{1}, \phi_{2}) = -(dv^{\sharp})_{0}(-\phi_{2}, -\phi_{1}), \quad (dv^{\flat})_{0}(\phi_{1}, \phi_{2}) = (dv^{\flat})_{0}(-\phi_{2}, -\phi_{1}).
\end{equation}
The identities \eqref{EquationSymmetriesForFProof} now follow from \eqref{EquationSymmetriesForcsharp0}, \eqref{EquationSymmetriesForcflat0} and \eqref{EquationSymmetriesFordv0}, which proves (ii) and (iii).
\end{proof}

Next we consider symmetries of the solutions of system \eqref{EquationSystemForf0}. We introduce some notation first. The~$\mathfrak{S}_{3}$-action on $(S^1)^{(3)}$ commutes with the $K$-action, whence it descends to an action on the domain $\Om$. To describe this action explicitly, we denote by $s_{1}$ and~$s_{2}$ the Coxeter generators of~$\mathfrak{S}_{3}$ that act on $(S^1)^{(3)}$ by swapping coordinates in the pairs $(\th_{0}, \th_{1})$ and $(\th_{1}, \th_{2})$, respectively. Then, with respect to the coordinates $(\phi_{1},\phi_{2})$ on $\Om$, the actions of $s_{1}$ and $s_{2}$ are given by
\[
s_{1}.(\phi_{1}, \phi_{2}) = (-\phi_{1}, \phi_{2}-\phi_{1}), \quad s_{2}.(\phi_{1}, \phi_{2}) = (\phi_{2}, \phi_{1}).
\]
A function $h_{0} \in C^{\infty}(\Omega)$ will be called \emph{alternating under the action of $\mathfrak{S}_{3}$} if $s.h_{0} = (-1)^{s} \, h_{0}$ for all $s \in \mathfrak{S}_{3}$. Thus a function $h_{0} \in C^{\infty}(\Om)$ is alternating under the action of $\mathfrak{S}_{3}$ if and only if the function $h \in C^\infty((S^1)^{(3)})^{K}$ defined by
\[
h(\th_{0}, \th_{1}, \th_{2}) = h_{0}(\th_{1}-\th_{0}, \th_{2}-\th_{0})
\]
is alternating in the usual sense.

\begin{proposition}[Alternating solutions] \label{PropositionAlternatingSolution}
	Assume that the cocycle $c$ is alternating. A solution $f_{0} \in C^{\infty}(\Om)$ of system \eqref{EquationSystemForf0} is alternating under the action of $\mathfrak{S}_{3}$ if and only if
\begin{equation} \label{EquationAlternatingInitialValues}
	f_{0}(\omega_{+}) = - f_{0}(\omega_{-}).
\end{equation}
In this case the primitive $P_{c}(f) \in \Pca(c)^{G}$, where $f$ is defined by \eqref{EquationCorrespondencefAndf0}, is alternating.
\end{proposition}

\begin{proof}
	First of all, we observe that $\mathfrak{S}_{3}$ acts on the base points $\{\omega_{+},\omega_{-}\}$ by
\begin{equation} \label{EquationActionOfS3Onomega}
	s_{1}.\omega_{\pm}=\omega_{\mp}, \quad s_{2}.\omega_{\pm}=\omega_{\mp}.
\end{equation}
Hence, if $f_{0}$ is alternating under the action of $\mathfrak{S}_{3}$ it follows that
\[
f_{0}(\omega_{+}) = (-1)^{s_{1}} \, f_{0}(s_{1}.\omega_{+}) = - f_{0}(\omega_{-}).
\]
Conversely, let $f_{0}$ be a solution of system \eqref{EquationSystemForf0} that satisfies \eqref{EquationAlternatingInitialValues}. We will prove that $f_{0}$ coincides with its antisymmetrization under the action of $\mathfrak{S}_{3}$. By Proposition \ref{PropositionExplicitPrimitives}\,(ii), the function $f_{0}$ corresponds to a solution $f \in C^{\infty}((S^{1})^{(3)})^{K}$ of system \eqref{EquationSystemForf} via
\[
f(\th_{0}, \th_{1}, \th_{2}) = f_{0}(\th_{1}-\th_{0}, \th_{2}-\th_{0}).
\]
Let now
\[
\hat{f} \deq \frac{1}{6} \cdot \sum_{s \in \mathfrak{S}_{3}} (-1)^{s} \, s.f
\]
be the antisymmetrization of $f$. Then $\hat{f} \in C^{\infty}((S^{1})^{(3)})^{K}$, and we further claim that $\hat{f}$ solves system \eqref{EquationSystemForf} as well. To see this, observe that in system \eqref{EquationSystemForf} the operators $L_{A}^{(3)}$ and $L_{N}^{(3)}$ are symmetric, while by Proposition~\ref{PropositionSymmetries}\,(i) the inhomogeneities $c^{\sharp} + dv^{\sharp}$ and $c^{\flat} + dv^{\flat}$ are alternating. Now by $K$-invariance, the function $\hat{f}$ gives rise to a function $\hat{f}_{0} \in C^{\infty}(\Om)$ via
\[
\hat{f}(\th_{0}, \th_{1}, \th_{2}) = \hat{f}_{0}(\th_{1}-\th_{0}, \th_{2}-\th_{0}).
\]
Then Proposition \ref{PropositionExplicitPrimitives}\,(ii) implies that $\hat{f}_{0}$ solves system \eqref{EquationSystemForf0}. Moreover, we have
\[
\hat{f}_{0} = \frac{1}{6} \cdot \sum_{s \in \mathfrak{S}_{3}} (-1)^{s} \, s.f_{0},
\]
whence $\hat{f}_{0}$ is alternating under the action of $\mathfrak{S}_{3}$. It follows from \eqref{EquationAlternatingInitialValues} and \eqref{EquationActionOfS3Onomega} that $\hat{f}_{0}(\omega_{\pm}) = f_{0}(\omega_{\pm})$. The uniqueness statement in Proposition \ref{PropositionExplicitPrimitives}\,(iii) implies that $\hat{f}_{0}$ coincides with $f_{0}$.
\end{proof}

The proposition shows that solutions $f_{0}$ of system \eqref{EquationSystemForf0} that are alternating under the action of $\mathfrak{S}_{3}$ form a $1$-parameter family.

\subsection{Boundedness}
\label{SubSectionBoundedness}

In order to complete the proof of Theorem \ref{TheoremMain} it remains to show that among the $G$-invariant primitives we constructed in Section \ref{SectionConstructionOfPrimitives}, there actually exist bounded ones. This is the content of the next proposition, which crucially relies on the symmetries unveiled in the previous subsection.

\begin{proposition}[Boundedness] \label{PropositionBoundednessOfPrimitives}
	Assume that the cocycle $c$ is alternating. Let $f_{0} \in C^{\infty}(\Om)$ be a solution of system \eqref{EquationSystemForf0} that is alternating under the action of $\mathfrak{S}_{3}$. Then the corresponding primitive $P_{c}(f) \in \Pca(c)^{G}$, where~$f$ is defined by \eqref{EquationCorrespondencefAndf0}, is bounded.
\end{proposition}

The proof of the proposition relies on the following three basic observations.

\begin{lemma} \label{LemmaBoundednessAlongLine}
	Let the function $f_{0}$ be defined by formula \eqref{EquationFormulaForf0}. If $f_{0}$ is bounded along the line segments
\[
(0,2\pi/3) \cup (2\pi/3,2\pi) \ni \xi \mapsto (2\pi/3,\xi)
\]
and
\[
(0,4\pi/3) \cup (4\pi/3,2\pi) \ni \xi \mapsto (4\pi/3,\xi),
\]
and $f$ is given by formula \eqref{EquationCorrespondencefAndf0}, then the corresponding primitive $P_{c}(f) \in \Pca(c)^{G}$ is bounded.
\end{lemma}

\begin{proof}
	By Proposition \ref{PropositionExplicitPrimitives}, $P_{c}(f) = I(c) + df$ is $G$-invariant. By $3$-transitivity of the $G$-action on~$S^{1}$ and since $I(c)$ is bounded, we therefore deduce that $P_{c}(f)$ is bounded if and only if the function
\[
z \mapsto df \bigl( 1, e^{2\pi i/3}, e^{4\pi i/3}, z \bigr)
\]
is bounded. Writing $z=e^{i\xi}$, we may express this function as
\[
\begin{split}
	\xi \mapsto & \,\, f(2\pi/3, 4\pi/3, \xi) - f(0, 4\pi/3, \xi) + f(0, 2\pi/3, \xi) - f(0, 2\pi/3, 4\pi/3) \\
	&= f_{0}(2\pi/3,\xi-2\pi/3) - f_{0}(4\pi/3,\xi) + f_{0}(2\pi/3,\xi) - f_{0}(2\pi/3,4\pi/3).
\end{split}
\]
The lemma follows.
\end{proof}

\begin{lemma} \label{LemmaBoundednessNearDiagonal}
	Let $C$ be a compact subset of the open square $(0, 2\pi)^2$. If the function $f_{0}$ defined by formula \eqref{EquationFormulaForf0} is bounded along the antidiagonal $\Deop$ in $\Om$, then it is bounded on the subset $C \cap \Om$ of $\Om$.
\end{lemma}

\begin{proof}
	By Lemma \ref{LemmaBoundednessOfrAndw} the function $F^{\flat}_{c} = c^{\flat}_{0} + (dv^{\flat})_{0}$ is bounded. Moreover, by assumption we have $|f_{0}|_{\Deop}| \le M$ for some constant $M>0$. Hence we obtain from formula \eqref{EquationFormulaForf0} the estimate
\[
\bigl| f_{0}(\phi_{1}, \phi_{2}) \bigr| \leq M + \bigl\| F^{\flat}_{c} \bigr\|_{\infty} \cdot \bigl| T(\phi_{1}, \phi_{2}) \bigr|
\]
for all $(\phi_{1},\phi_{2}) \in \Om$. It remains to show that the function $T$ is bounded on $C \cap \Om$. By compactness of $C$ it will be enough to prove that the function $\map{T}{\Om}{\R}$ extends to a continuous function on the open square $(0,2\pi)^2$. This, however, is immediate from formula \eqref{EquationDefinitionOfT}.
\end{proof}

\begin{figure}
	\includegraphics{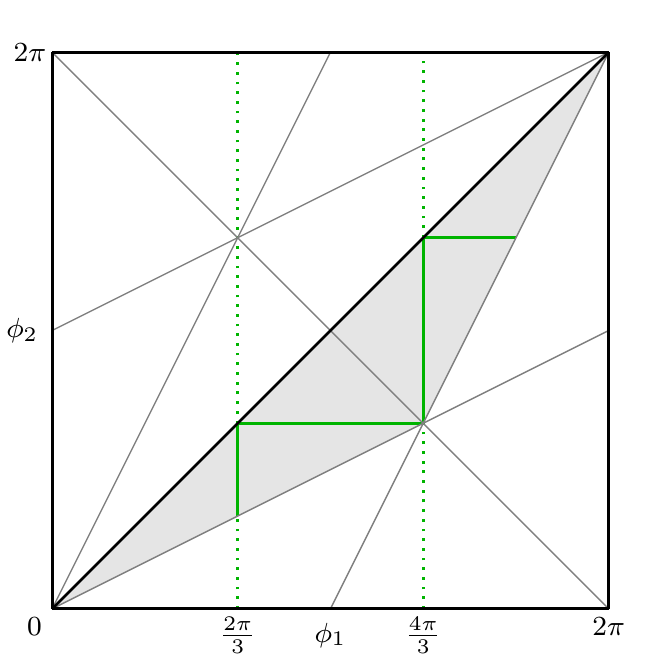}
	\centering
	\caption{A fundamental domain for the $\mathfrak{S}_{3}$-action on $\Om$ (shaded), and the images of the line segments $\xi \mapsto (2\pi/3,\xi)$ and $\xi \mapsto (4\pi/3,\xi)$ therein under the $\mathfrak{S}_{3}$-action}
	\label{FigureLines}
\end{figure}

\begin{lemma} \label{LemmaConstantAlongAntidiagonal}
	Assume that the cocycle $c$ is alternating. Then the function $f_{0}$ defined by formula \eqref{EquationFormulaForf0} is locally constant along the antidiagonal $\Deop$ in $\Om$.
\end{lemma}

\begin{proof}
	Since $c$ is alternating, the inhomogeneity $F^{\sharp}_{c}$ vanishes along the antidiagonal $\Deop$ by Proposition \ref{PropositionSymmetries}\,(ii). The lemma now follows from formula \eqref{EquationFormulaForf0}.
\end{proof}

\begin{example}
	Assume that the cocycle $c$ is alternating. Consider the special solution~$f_{0}$ determined by the initial values $f_{0}(\omega_{\pm})=0$. It is alternating under the action of $\mathfrak{S}_{3}$ by Proposition~\ref{PropositionAlternatingSolution}. Moreover, by Lemma~\ref{LemmaConstantAlongAntidiagonal} it vanishes along the antidiagonal~$\Deop$. Since under the action of $\mathfrak{S}_{3}$ the components of $\Deop$ get identified with the medians of the triangles enclosing the domains $\Om_{+}$ and $\Om_{-}$, we further infer that~$f_{0}$ also vanishes along these medians. Moreover, by Proposition~\ref{PropositionSymmetries}\,(iii) the function $F^{\flat}_{c}$ is symmetric about the antidiagonal. Thus we see from formula \eqref{EquationFormulaForf0} that the special solution~$f_{0}$ is antisymmetric with respect to the antidiagonal, and hence antisymmetric with respect to all medians.
\end{example}

We are now ready to prove Proposition \ref{PropositionBoundednessOfPrimitives}.

\begin{proof}[Proof of Proposition \ref{PropositionBoundednessOfPrimitives}]
	By Proposition \ref{PropositionExplicitPrimitives}\,(iii) the function $f_{0}$ is given by formula \eqref{EquationFormulaForf0}. Hence by Lemma~\ref{LemmaBoundednessAlongLine} it suffices to show that $f_{0}$ is bounded along the line segments $\xi \mapsto (2\pi/3,\xi)$ and $\xi \mapsto (4\pi/3,\xi)$. Since $f_{0}$ is alternating under the action of $\mathfrak{S}_{3}$, it suffices to prove that $f_{0}$ is bounded along the images of these line segments in any fundamental domain for the $\mathfrak{S}_{3}$-action on $\Om$. In fact, we may choose the fundamental domain in such a way that the image line segments lie inside a compact subset of $(0, 2\pi)^2$ (see Figure \ref{FigureLines}). By Lemma \ref{LemmaConstantAlongAntidiagonal} and Lemma \ref{LemmaBoundednessNearDiagonal}, the function $f_{0}$ is then bounded on these line segments.
\end{proof}

Theorem \ref{TheoremMain} follows by combining Lemma \ref{LemmaAlternatingCocycle} and Propositions \ref{PropositionExplicitPrimitives}, \ref{PropositionAlternatingSolution} and \ref{PropositionBoundednessOfPrimitives}.

\begin{appendix}

\section{The Frobenius integrability condition}
\label{SectionTheFrobeniusIntegrabilityCondition}

The goal of this appendix is to prove Proposition \ref{PropositionIntegrabilityCondition}. We have to show that the system
\begin{equation} \label{EquationSystemForfAppendix}
	\left\{
	\begin{aligned}
		L_{K}^{(3)}f &=  0 \\
		L_{A}^{(3)}f &= c^{\sharp} + dv^{\sharp} \\
		L_{N}^{(3)}f &= c^{\flat} + dv^{\flat}
	\end{aligned}
	\right.
\end{equation}
admits a solution $(f,v^{\sharp},v^{\flat})$ if and only if the pair $(v^{\sharp},v^{\flat})$ satisfies the Frobenius system \eqref{EquationFrobeniusSystem}. Here we consider $f$ and $v^{\sharp}$, $v^{\flat}$ as smooth functions on the domains $D \deq [0,2\pi)^{(3)}$ and $[0,2\pi)^{(2)}$, respectively. It will be convenient to replace system \eqref{EquationSystemForfAppendix} by the  the equivalent system
\begin{equation} \label{EquationModifiedSystemForfAppendix}
	\left\{
	\begin{aligned}
		L_{K}^{(3)}f &=  0 \\
		L_{A}^{(3)}f &= c^{\sharp} + dv^{\sharp} \\
		L_{N}^{(3)}f - L_{K}^{(3)}f &= c^{\flat} + dv^{\flat}.
	\end{aligned}
	\right.
\end{equation}
Consider the product $D \times \R$. We denote the coordinates on $D$ by $(\th_{0},\th_{1},\th_{2})$ and the coordinate on $\R$ by $\th_{3}$. The graph $\Gamma_{f} \deq \{ ( (\th_{0},\th_{1},\th_{2}),f(\th_{0},\th_{1},\th_{2}) ) \}$ of the function $f$ is a $3$-dimensional submanifold of $D \times \R$. Define vector fields $X$, $Y$, $Z$ on $D \times \R$ by
\[
\begin{split}
	{X} &\deq L_{K}^{(3)}, \\
	{Y} &\deq L_{A}^{(3)} + \bigl( c^{\sharp} + dv^{\sharp} \bigr) \, \partial_{\th_{3}}, \\
	{Z} &\deq L_{N}^{(3)} - L_{K}^{(3)} + \bigl( c^{\flat} + dv^{\flat} \bigr) \, \partial_{\th_{3}}.
\end{split}
\]
Since $G$ acts strictly $3$-transitively on $D$, it follows that these vector fields span a distribution~$E$ of constant rank $3$ on $D \times \R$. Then a triple $(f,v^{\sharp},v^{\flat})$ is a solution of system \eqref{EquationModifiedSystemForfAppendix} if and only if the graph $\Gamma_{f}$ is an integral manifold for $E$. Hence the Frobenius theorem (see e.\,g.\,\cite[Ch.\,11.]{Lee/Manifolds-and-differential-geometry}) implies that system \eqref{EquationModifiedSystemForfAppendix} admits a solution $(f,v^{\sharp},v^{\flat})$ if and only if the distribution $E$ is integrable, i.\,e.\,, the vector fields $X$, $Y$, $Z$ form an involutive system. Note that
\[
\left[ L_{K}^{(3)}, L_{A}^{(3)} \right] = L_{K}^{(3)} - L_{N}^{(3)}, \quad \left[ L_{K}^{(3)}, L_{N}^{(3)} - L_{K}^{(3)} \right] = L_{A}^{(3)}, \quad \left[ L_{A}^{(3)}, L_{N}^{(3)} - L_{K}^{(3)} \right] = L_{K}^{(3)}.
\]
Hence the vector fields $X$, $Y$, $Z$ form an involutive system if and only if
\begin{equation} \label{EquationFrobeniusCondition}
	[X,Y] = -Z, \quad [X,Z] = Y, \quad [Y,Z] = X.
\end{equation}
We shall now make these conditions explicit. We start with two preliminary lemmas.

\begin{lemma} \label{LemmaKActionOncSharpAndcFlat}
	The functions $c^\sharp$ and $c^\flat$ defined in \eqref{EquationDefinitionOfcsharp} and \eqref{EquationDefinitionOfcflat} satisfy
\[
L_{K}^{(3)} c^{\sharp} = -c^{\flat}, \quad L_{K}^{(3)}c^{\flat} =  c^{\sharp}.
\]
\end{lemma}

\pagebreak

\begin{proof}
	More generally, we prove that for any function $\lambda \in C^\infty((0,2\pi))$, the function
\[
c_{\lambda}(\th_{0}, \th_{1}, \th_{2}) \deq \fint \fint \lambda(\varphi) \, c(\eta, \varphi, \th_{0}, \th_{1}, \th_{2}) \, d\eta \, d\varphi
\]
satisfies $L_{K}^{(3)} c_{\lambda} = c_{\lambda'}$. Indeed, by $K$-invariance of the cocycle $c$ and the measure, we have
\[
\begin{split}
	L_{K}^{(3)}{c_{\lambda}}(\th_{0}, \th_{1}, \th_{2}) &= \left.\frac{d}{d\xi}\right|_{\xi=0} \fint \fint \lambda(\varphi) \, c(\eta, \varphi, \th_{0}+\xi, \th_{1}+\xi, \th_{2}+\xi) \, d\eta \, d\varphi \\
	&= \left.\frac{d}{d\xi}\right|_{\xi=0} \fint \fint \lambda(\varphi) \, c(\eta-\xi, \varphi-\xi, \th_{0}, \th_{1}, \th_{2}) \, d\eta \, d\varphi \\
	&= \fint \fint \left.\frac{d}{d\xi}\right|_{\xi=0} \lambda(\varphi+\xi) \, c(\eta, \varphi, \th_{0}, \th_{1}, \th_{2}) \, d\eta \, d\varphi \\
	&= \fint \fint \lambda'(\varphi) \, c(\eta, \varphi, \th_{0}, \th_{1}, \th_{2}) \, d\eta \, d\varphi. \qedhere
\end{split}
\]
\end{proof}

\begin{lemma} \label{LemmaDerivativeOfccheck}
	The function $\check{c}$ defined in \eqref{EquationDefinitionOfccheck} satisfies
\begin{equation} \label{EquationDerivativeOfccheck}
	L_{K}^{(3)} c^{\sharp} - L_{N}^{(3)} c^{\sharp} + L_{A}^{(3)} c^{\flat} = - d \check{c}.
\end{equation}
\end{lemma}

\begin{proof}
	Let us consider the left-hand side of \eqref{EquationDerivativeOfccheck}. In a first step, using $G$-invariance of $c$ and $K$-invariance of the measure, we compute
\[
\begin{split}
	L_{K}^{(3)} c^{\sharp}(\th_{0}, \th_{1}, \th_{2}) &= \left.\frac{d}{d\xi} \right|_{\xi=0} \fint \fint \cos(\varphi) \, c(\eta, \varphi, \th_{0}+\xi, \th_{1}+\xi, \th_{2}+\xi) \, d\eta \, d\varphi \\
	&= \fint \fint \left.\frac{d}{d\xi} \cos(\varphi+\xi) \right|_{\xi=0} c(\eta, \varphi, \th_{0}, \th_{1}, \th_{2}) \, d\eta \, d\varphi
\end{split}
\]
and
\[
\begin{split}
	L_{N}^{(3)} c^{\sharp}(\th_{0}, \th_{1}, \th_{2}) &= \left.\frac{d}{dt} \right|_{t=0} \fint \fint \cos(\varphi) \, c(\eta, \varphi, n_{t}.\th_{0}, n_{t}.\th_{1}, n_{t}.\th_{2}) \, d\eta \, d\varphi \\
	&= \fint \fint \left.\frac{d}{dt} \left( \cos(n_{t}.\varphi) \, \frac{d(n_{t}.\varphi)}{d\varphi} \frac{d(n_{t}.\eta)}{d\eta} \right) \right|_{t=0} c(\eta, \varphi, \th_{0}, \th_{1}, \th_{2}) \, d\eta \, d\varphi,
\end{split}
\]
and similarly
\[
\begin{split}
	L_{A}^{(3)} c^{\flat}(\th_{0}, \th_{1}, \th_{2}) &= \left.\frac{d}{ds} \right|_{s=0} \fint \fint \sin(\varphi) \, c(\eta, \varphi, a_{s}.\th_{0}, a_{s}.\th_{1}, a_{s}.\th_{2}) \, d\eta \, d\varphi \\
	&= \fint \fint  \left.\frac{d}{ds} \left( \sin(a_{s}.\varphi) \, \frac{d(a_{s}.\varphi)}{d\varphi} \frac{d(a_{s}.\eta)}{d\eta} \right) \right|_{s=0} c(\eta, \varphi, \th_{0}, \th_{1}, \th_{2}) \, d\eta \, d\varphi.
\end{split}
\]

\noindent Second, using Lemma \ref{LemmaInfinitesimalANAction} we compute the derivatives appearing in the above formulas. Firstly, we have $\left.\frac{d}{d\xi}(\cos(\varphi+\xi))\right|_{\xi=0} = -\sin(\varphi)$. Moreover,

\pagebreak

\[
\begin{split}
	&\left.\frac{d}{dt} \! \left( \cos(n_{t}.\varphi) \, \frac{d(n_{t}.\varphi)}{d\varphi} \frac{d(n_{t}.\eta)}{d\eta} \right) \right|_{t=0} \\
	&= \left. -\sin(n_{t}.\varphi) \, \frac{d(n_{t}.\varphi)}{dt} \, \frac{d(n_{t}.\varphi)}{d\varphi} \frac{d(n_{t}.\eta)}{d\eta} \right|_{t=0} + \left.\cos(n_{t}.\varphi) \, \frac{d}{d\varphi} \frac{d(n_{t}.\varphi)}{dt} \frac{d(n_{t}.\eta)}{d\eta} \right|_{t=0} \\
	& \hspace{2cm} {} + \left.\cos(n_{t}.\varphi) \, \frac{d(n_{t}.\varphi)}{d\varphi} \frac{d}{d\eta} \frac{d(n_{t}.\eta)}{dt} \right|_{t=0} \\
	&= \left. -\sin(n_{t}.\varphi) \, \bigl( 1-\cos(n_{t}.\varphi) \bigr) \, \frac{d(n_{t}.\varphi)}{d\varphi} \frac{d(n_{t}.\eta)}{d\eta} \right|_{t=0} \\
	& \hspace{2cm} {} + \left.\cos(n_{t}.\varphi) \, \frac{d}{d\varphi}\bigl( 1-\cos(n_{t}.\varphi) \bigr) \, \frac{d(n_{t}.\eta)}{d\eta} \right|_{t=0} \\
	& \hspace{5cm} {} + \left.\cos(n_{t}.\varphi) \, \frac{d(n_{t}.\varphi)}{d\varphi} \frac{d}{d\eta}\bigl( 1-\cos(n_{t}.\eta) \bigr) \right|_{t=0} \\
	&= -\sin(\varphi) \, (1-\cos(\varphi)) + \cos(\varphi)\sin(\varphi) + \cos(\varphi)\sin(\eta) \\
	&= 2\sin(\varphi) \cos(\varphi) + \cos(\varphi) \sin(\eta) - \sin(\varphi).
\end{split}
\]
Lastly,
\[
\begin{split}
	& \left.\frac{d}{ds} \! \left( \sin(a_{s}.\varphi) \frac{d(a_{s}.\varphi)}{d\varphi} \frac{d(a_{s}.\eta)}{d\eta} \right) \right|_{s=0} \\
	&= \left.\cos(a_{s}.\varphi) \, \frac{d(a_{s}.\varphi)}{ds} \frac{d(a_{s}.\varphi)}{d\varphi} \frac{d(a_{s}.\eta)}{d\eta} \right|_{s=0} + \left.\sin(a_{s}.\varphi) \, \frac{d}{d\varphi} \frac{d(a_{s}.\varphi)}{ds} \frac{d(a_{s}.\eta)}{d\eta} \right|_{s=0} \\
	& \hspace{8cm} {} + \left.\sin(a_{s}.\varphi) \, \frac{d(a_{s}.\varphi)}{d\varphi} \frac{d}{d\eta} \frac{d(a_{s}.\eta)}{ds} \right|_{s=0} \\
	&= \left.\cos(a_{s}.\varphi) \sin(a_{s}.\varphi) \, \frac{d(a_{s}.\varphi)}{d\varphi} \frac{d(a_{s}.\eta)}{d\eta} \right|_{s=0} + \left.\sin(a_{s}.\varphi) \frac{d}{d\varphi} \sin(a_{s}.\varphi) \, \frac{d(a_{s}.\eta)}{d\eta} \right|_{s=0} \\
	& {} \hspace{8cm} + \left.\sin(a_{s}.\varphi) \, \frac{d(a_{s}.\varphi)}{d\varphi} \, \frac{d}{d\eta} \sin(a_{s}.\eta) \right|_{s=0} \\
	&= 2 \sin(\varphi) \cos(\varphi) + \sin(\varphi) \cos(\eta).
\end{split}
\]
Summing up, we obtain
\begin{eqnarray} \label{EquationFrobeniusLemmaLHS} 
\begin{aligned}
	& L_{K}^{(3)} c^{\sharp}(\th_{0}, \th_{1}, \th_{2}) - L_{N}^{(3)} c^{\sharp}(\th_{0}, \th_{1}, \th_{2}) + L_{A}^{(3)} c^{\flat}(\th_{0}, \th_{1}, \th_{2}) \\
	&= \fint \fint \bigl( \sin(\varphi)\cos(\eta) - \cos(\varphi)\sin(\eta) \bigr) \, c(\eta, \varphi, \th_{0}, \th_{1}, \th_{2}) \, d\eta \, d\varphi \\
	&= - \fint \fint \sin(\eta-\varphi) \, c(\eta, \varphi, \th_{0}, \th_{1}, \th_{2}) \, d\eta \, d\varphi.
\end{aligned}
\end{eqnarray}
Now we turn to the computation of the right-hand side of \eqref{EquationDerivativeOfccheck}. The cocycle identity for $c$ yields

\pagebreak

\begin{multline*}
	0 = dc(\eta, \varphi, \psi, \th_{0}, \th_{1}, \th_{2}) = c(\varphi, \psi, \th_{0}, \th_{1}, \th_{2}) - c(\eta, \psi,\th_{0}, \th_{1}, \th_{2}) + c(\eta, \varphi, \th_{0}, \th_{1}, \th_{2}) \\ -\sum_{j=0}^{2} (-1)^{j} \, c(\eta, \varphi, \psi, \th_{0}, \ldots, \widehat{\th_{j}}, \ldots,  \th_{2}).
\end{multline*}
We multiply this identity by $\sin(\eta-\varphi)$ and integrate over the variables $\eta, \varphi, \psi$. Integrating the first term, we get
\[
\begin{split}
	& \fint\fint\fint \sin(\eta-\varphi) \, c(\varphi, \psi, \th_{0}, \th_{1}, \th_{2}) \, d\eta \, d\varphi \, d\psi \\
	&= \fint\fint \left(  \fint \sin(\eta-\varphi) \, d\eta \right) c(\varphi, \psi, \th_{0}, \th_{1}, \th_{2}) \, d\varphi \, d\psi \, = \, 0.
\end{split}
\]
Likewise, the integral of the second term vanishes. We are thus left with
\[
\begin{split}
	& \fint \fint \sin(\eta-\varphi) \, c(\eta, \varphi, \th_{0}, \th_{1}, \th_{2}) \, d\eta \, d\varphi \\
	&= \fint \fint \fint \sin(\eta-\varphi) \left( \sum_{j=0}^{2} (-1)^{j} \, c(\eta, \varphi, \psi, \th_{0}, \ldots, \widehat{\th_{j}}, \ldots,  \th_{2}) \right) d\eta \, d\varphi \, d\psi \\
	&= \sum_{j=0}^{2} (-1)^{j} \fint \fint \fint \sin(\eta-\varphi) \, c(\eta, \varphi, \psi, \th_{0}, \ldots, \widehat{\th_{j}}, \ldots,  \th_{2}) \, d\eta \, d\varphi \, d\psi = d\check{c}(\th_{0}, \th_{1}, \th_{2}).
\end{split}
\]
Comparing this with \eqref{EquationFrobeniusLemmaLHS} above, formula \eqref{EquationDerivativeOfccheck} follows.
\end{proof}

We are now in a position to finish the proof of Proposition \ref{PropositionIntegrabilityCondition} by spelling out the integrability conditions \eqref{EquationFrobeniusCondition}. Consider the first identity in \eqref{EquationFrobeniusCondition}. We have
\[
[X,Y] = \bigl[ L_{K}^{(3)}, L_{A}^{(3)} \bigr] + \bigl( L_{K}^{(3)} (c^{\sharp} +dv^{\sharp}) \bigr) \partial_{\th_{3}} = L_{K}^{(3)} - L_{N}^{(3)} + \bigl( L_{K}^{(3)} (c^{\sharp} +dv^{\sharp}) \bigr) \partial_{\th_{3}}.
\]
Recall from Lemma \ref{LemmaKActionOncSharpAndcFlat} and Lemma \ref{LemmaOperatorsCommuteWithDifferentials} that
\[
L_K^{(3)}c^\sharp = -c^\flat, \quad L_K^{(3)}dv^\sharp = dL_K^{(2)}v^\sharp.
\]
Thus
\[
[X,Y] = L_{K}^{(3)} - L_{N}^{(3)} + \bigl( -c^\flat+dL_K^{(2)}v^\sharp \bigr) \partial_{\th_{3}}.
\]
Comparing this to $-{Z}$ we find
\begin{equation} \label{EquationFrobenius1}
	[X,Y] =-Z \quad \Longleftrightarrow \quad d\bigl( L_{K}^{(2)} v^{\sharp} + v^{\flat} \bigr) = 0.
\end{equation}
Likewise, for the second identity in \eqref{EquationFrobeniusCondition} we have
\begin{equation}\label{EquationFrobenius2}
	[X,Z] = Y \quad \Longleftrightarrow \quad d\bigl( L_{K}^{(2)} v^{\flat} - v^{\sharp} \bigr) = 0.
\end{equation}
Finally, observe that 
\[
\begin{split}
	[Y,Z] &= \bigl[ L_{A}^{(3)},L_{N}^{(3)}-L_{K}^{(3)} \bigr] + \bigl[ L_{A}^{(3)},( c^{\flat}+dv^{\flat} ) \partial_{\th_{3}} \bigr] - \bigl[ L_{N}^{(3)} - L_{K}^{(3)},( c^{\sharp}+dv^{\sharp} ) \partial_{\th_{3}} \bigr] \\
	&= L_{K}^{(3)} + \left( \bigl( L_{K}^{(2)} - L_{N}^{(2)} \bigr) dv^{\sharp} + L_{A}^{(3)} dv^\flat + \bigl( L_{K}^{(3)} - L_{N}^{(3)} \bigr) c^{\sharp} + L_{A}^{(3)} c^{\flat} \right) \! \partial_{\th_{3}}.
\end{split}
\]
By Lemma \ref{LemmaDerivativeOfccheck} and Lemma \ref{LemmaOperatorsCommuteWithDifferentials}, this becomes
\[
[Y,Z] = L_{K}^{(3)} + \left( d\bigl( L_{K}^{(2)} - L_{N}^{(2)} \bigr) v^{\sharp} + d L_{A}^{(2)} v^{\flat} - d\check{c} \right) \! \partial_{\th_{3}}.
\]
We deduce that
\begin{equation} \label{EquationFrobenius3}
	[Y,Z] = X \quad \Longleftrightarrow \quad d\bigl( L_{K}^{(2)} v^{\sharp} - L_{N}^{(2)} v^{\sharp} + L_{A}^{(2)} v^{\flat} - \check{c} \bigr) = 0.
\end{equation}
Combining \eqref{EquationFrobenius1}, \eqref{EquationFrobenius2} and \eqref{EquationFrobenius3}, Proposition \ref{PropositionIntegrabilityCondition} now follows from \eqref{EquationFrobeniusCondition}.

\end{appendix}

\bibliographystyle{amsplain}

\end{document}